\documentclass[11pt]{amsart}

\usepackage{a4wide}
\usepackage{amssymb}
\usepackage{mathrsfs}
\usepackage{enumerate}
\usepackage{esint}
\usepackage{titletoc}
\usepackage{mathscinet}
\usepackage[colorlinks=true, urlcolor=blue, linkcolor=blue, citecolor=green]{hyperref}
\usepackage{mathtools}
\usepackage{xcolor}
\usepackage[nameinlink]{cleveref}

\theoremstyle{plain}
\newtheorem{theorem}{Theorem}[section]
\newtheorem{lemma}[theorem]{Lemma}
\newtheorem{corollary}[theorem]{Corollary}
\newtheorem{proposition}[theorem]{Proposition}

\theoremstyle{definition}
\newtheorem{definition}[theorem]{Definition}
\newtheorem{remark}[theorem]{Remark}

\numberwithin{equation}{section}

\DeclareMathOperator*{\tr}{tr}

\DeclareMathOperator*{\dist}{dist}
\DeclareMathOperator*{\osc}{osc}

\makeatletter
\@namedef{subjclassname@2020}{\textup{2020} Mathematics Subject Classification}
\makeatother

\title[Logarithmic Hessian estimates]{Logarithmic Hessian estimates for non-uniformly elliptic equations in non-divergence form}

\author{Jongmyeong Kim}
\address{Institute of Mathematics, Academia Sinica, Taipei 106319, Taiwan}
\email{jmkim@gate.sinica.edu.tw}

\author{Se-Chan Lee}
\address{School of Mathematics, Korea Institute for Advanced Study, 02455 Seoul, Republic of Korea}
\email{sechan@kias.re.kr}



\begin{document}
\begin{abstract}
  We establish uniform logarithmic Hessian estimates for solutions of non-uniformly elliptic equations in non-divergence form. The proof is based on a measure estimate from convex analysis without localization or covering arguments. Moreover, we construct a counterexample that rules out uniform power-type Hessian estimates and shows the sharpness of the logarithmic exponent in the plane. Finally, we extend our estimates to strong solutions and apply them to obtain regularity results for Grushin-type equations.
\end{abstract}

\maketitle

\section{Introduction}\label{sec-introduction}
The regularity theory for second-order partial differential equations has been extensively developed in both divergence and non-divergence form. For linear uniformly elliptic equations in divergence form, the De Giorgi--Nash--Moser theory~\cite{DGio57,Nas58,Mos60,Mos61} yields the H\"older continuity of weak solutions. For the corresponding linear equations in non-divergence form, the Aleksandrov--Bakelman--Pucci (ABP) estimate~\cite{Ale66,Bak61,Puc66} and the Krylov--Safonov theory~\cite{KS79,KS80} provide the maximum principle, Harnack inequality, and H\"older estimates. Since these arguments essentially rely on ellipticity rather than linearity, they have become fundamental tools in the theory of fully nonlinear uniformly elliptic equations; we refer to the monographs~\cite{CC95, GT01}. 

The next natural question concerns the Sobolev-type regularity of solutions. Lin~\cite{Lin86} first proved that solutions of linear uniformly elliptic equations in non-divergence form satisfy a $W^{2,\varepsilon}$ estimate for some small universal exponent $\varepsilon>0$. Caffarelli~\cite{Caf89} used this $W^{2, \varepsilon}$ estimate as the starting point of a Calder\'on--Zygmund perturbation argument, and obtained $W^{2, p}$ estimates for fully nonlinear convex equations when $p>n$, under suitable small oscillation assumptions on the operator. We refer to the works of Wang~\cite{Wan92, Wan92a} for the parabolic counterpart. 

Moreover, Armstrong--Silvestre--Smart~\cite{ASS12}, based on the small perturbation theory of Savin~\cite{Sav07}, showed that the singular set of a solution of fully nonlinear equations has Hausdorff dimension at most $n-\varepsilon$. This observation describes the importance of the quantitative bounds for $\varepsilon$. Later, Mooney~\cite{Moo19} provided a proof of the Krylov--Safonov theorem without localization, and obtained the quantitative lower bound $\varepsilon \geq c(n)(\lambda/\Lambda)^{n-1}$ for uniformly elliptic equations with ellipticity constants $0<\lambda \leq \Lambda$; see also \cite{NT23}. We point out that all of these results are concerned with \emph{uniformly elliptic} equations.\medskip

The goal of this paper is to develop uniform logarithmic Hessian estimates for \emph{non-uniformly} elliptic equations in non-divergence form. These estimates can be viewed as a non-uniformly elliptic analogue of Lin's classical $W^{2, \varepsilon}$ estimate. To illustrate the setting, let $B_1$ denote the unit ball in $\mathbb{R}^n$ and consider the linear equation
\begin{equation*}
    a_{ij}(x)D_{ij}u=f \quad \text{in $B_1$},
\end{equation*}
where, after normalization, the continuous ellipticity function $\lambda : B_1 \to (0,1]$ satisfies
\begin{equation*}
    \lambda(x)|\xi|^2 \leq a_{ij}(x)\xi_i\xi_j \leq |\xi|^2 \quad \text{for every $\xi \in \mathbb{R}^n$}.
\end{equation*}
Rather than imposing a uniform bound on $\|\lambda^{-1}\|_{L^\infty(B_1)}$, we quantify the degeneracy through
\begin{equation*}
    \|\lambda^{-1}\|_{L^p(B_1)} \quad \text{for some $p>n-1$}.
\end{equation*}
Under the assumption $\lambda^{-1} \in L^p(B_1)$ and a suitable integrability condition on $f/\lambda$, we prove that the Hessian of solutions enjoys logarithmic integrability of order $p/(n-1)$, with estimates independent of any pointwise lower bound for $\lambda$.

Let us briefly review the literature on non-uniformly elliptic equations. In divergence form, Trudinger~\cite{Tru71,Tru73} established local estimates under integrability assumptions on the ellipticity. Recently, Bella--Sch\"affner~\cite{BS21} presented sharp integrability conditions on the ellipticity that guarantee local boundedness and the Harnack inequality. We also mention that Fabes--Kenig--Serapioni~\cite{FKS82} derived local H\"older estimates, provided that the coefficients satisfy an $A_2$-Muckenhoupt weight condition. These arguments, however, rely heavily on variational methods such as Moser iteration, and so do not directly extend to equations in non-divergence form. 

In non-divergence form, Armstrong--Smart \cite{AS14} established an oscillation decay estimate and stochastic homogenization results. Their proof of oscillation decay combines localization with a measure-theoretic argument based on the Besicovitch covering theorem. More recently, the authors~\cite{KL25} established the ABP estimate for $L^n$-strong solutions of non-uniformly elliptic equations in non-divergence form. Finally, Bowman~\cite{Bow26} proved a logarithmic weak Harnack inequality under the sharp condition $p>n-1$, and the Harnack inequality when $p$ is larger than some dimensional constant.

\subsection{Main result}
Let $\lambda : B_1 \to (0, 1]$ be a continuous ellipticity function. We define the \emph{Pucci extremal operators} $\mathcal{M}_{\lambda, 1}^{\pm}$ (which are the normalized version of those in \cite{KL25}) by
\begin{equation*}
    \begin{aligned}
        \mathcal{M}_{\lambda, 1}^+(M)(x)&\coloneqq \sum_{e_i \geq 0} e_i(M) +\lambda(x) \sum_{e_i<0} e_i(M),\\
	\mathcal{M}_{\lambda, 1}^-(M)(x)&\coloneqq \lambda(x) \sum_{e_i \geq 0} e_i(M) +\sum_{e_i<0} e_i(M), 
    \end{aligned}
\end{equation*}
where $x \in B_1$, $M \in \mathcal{S}^n \coloneqq \{M \in \mathbb{R}^{n \times n} : M^T=M\}$, and $e_i(M)$'s are the eigenvalues of $M$. For a constant ellipticity $\lambda(\cdot) \equiv \lambda_0 \in (0, 1]$, they reduce to the classical Pucci extremal operators; see \cite{CC95, CIL92} for instance.

Our main result is a logarithmic version of the classical $W^{2, \varepsilon}$ estimate. We refer to Section~\ref{sec-preliminaries} for the definitions of the curvature quantity $\underline{\Theta}_u(x; B_1)$ and $C$-viscosity solutions.

\begin{theorem}[Global logarithmic Hessian estimate]\label{thm-log-hessian-estimate}
    Let $f,\lambda\in C(B_1)$ satisfy   
    \begin{equation*}
        0<\lambda \leq 1, \quad\lambda^{-1} \in L^p(B_1),  \quad \frac{f^+}{\lambda} \in L^{\sigma}(B_1)
    \end{equation*}
    for some $p>n-1$ and $\sigma=\max\{p/(p-n+1), n\}$. Suppose that $u\in C(B_1)\cap L^{\infty}(B_1)$ is a $C$-viscosity supersolution of 
   \begin{equation*}
       \mathcal{M}_{\lambda, 1}^-(D^2u) \leq f \quad  \text{in $B_1$}.
  \end{equation*}
Then there exists a constant $C>0$, depending only on $n$ and $p$, such that 
\begin{equation*}
    \begin{aligned}
         \int_{B_1} \log^{\frac{p}{n-1}} (1+ \underline{\Theta}_u(x; B_1)) \,\mathrm{d}x &\leq C(1+\|\lambda^{-1}\|_{L^p(B_1)}^p)\\
         &\quad+C\log^{\frac{p}{n-1}}\left(1+\|u\|_{L^{\infty}(B_1)}+\left\|\frac{f^+}{\lambda}\right\|_{L^{\sigma}(B_1)}\right).
    \end{aligned}
\end{equation*}
\end{theorem}

The constant $C>0$ in Theorem~\ref{thm-log-hessian-estimate} is independent of the pointwise lower bound of $\lambda$. This is the essential difference from the uniformly elliptic theory: the degeneracy is encoded in the $L^p$ norm of $\lambda^{-1}$. The estimate holds on the whole ball $B_1$, without any boundary condition or continuity assumption on $u$, $f$ or $\lambda$ at $\partial B_1$. 

An application of Theorem~\ref{thm-log-hessian-estimate} to both $u$ and $-u$ implies the following consequence.

\begin{corollary}\label{cor-log-hessian}
    Let $f,\lambda\in C(B_1)$ satisfy   
    \begin{equation*}
        0<\lambda \leq 1, \quad\lambda^{-1} \in L^p(B_1),  \quad \frac{f}{\lambda} \in L^{\sigma}(B_1)
    \end{equation*}
    for some $p>n-1$ and $\sigma=\max\{p/(p-n+1), n\}$. Suppose that $u\in C(B_1)\cap L^{\infty}(B_1)$ is a $C$-viscosity solution of 
   \begin{equation*}
     \mathcal{M}_{\lambda, 1}^+(D^2u) \geq -|f|  \quad \text{and} \quad \mathcal{M}_{\lambda, 1}^-(D^2u) \leq |f| \quad  \text{in $B_1$}.
  \end{equation*}
Then $u$ is twice differentiable almost everywhere in $B_1$, and there exists a constant $C>0$, depending only on $n$ and $p$, such that 
\begin{equation*}
    \begin{aligned}
         \int_{B_1} \log^{\frac{p}{n-1}} (1+\|D^2u\|) \,\mathrm{d}x &\leq C(1+\|\lambda^{-1}\|_{L^p(B_1)}^p)\\
         &\quad+C\log^{\frac{p}{n-1}}\left(1+\|u\|_{L^{\infty}(B_1)}+\left\|\frac{f}{\lambda}\right\|_{L^{\sigma}(B_1)}\right).
    \end{aligned}
\end{equation*}
\end{corollary}
Here we say that $u$ is twice differentiable at $x_0 \in B_1$ if $u$ admits a second-order Taylor expansion at $x_0$.

Let us briefly describe the strategy of the proof of Theorem~\ref{thm-log-hessian-estimate}. Following the global approach of Mooney~\cite{Moo19}, we use the convex envelopes and contact sets introduced in Section~\ref{sec-preliminaries}. The geometric arguments identify the new contact points obtained by sliding paraboloids and compare the measure of the set where the paraboloids touch the convex envelope with that of the corresponding vertices. Combining these arguments with the supersolution inequality and the area formula provides the measure estimate in Lemma~\ref{lem-nonhomo-measure-est}, which quantifies the gain between consecutive contact sets in terms of $\lambda^{-1}$ and $f^+/\lambda$. This lemma is the key ingredient in the iteration argument: after controlling the boundary term, we apply it at increasing openings and use H\"older's inequality together with a weighted summation to obtain the global logarithmic Hessian estimate. In our argument, the boundary term only contributes to the constant $C$ in the estimate and does not reduce the logarithmic exponent $p/(n-1)$. Thus, the global estimate also yields the interior estimate with the same exponent; see Remark~\ref{rmk-interior-global}.

\subsection{Failure of power-type estimates}
The logarithmic character of Theorem~\ref{thm-log-hessian-estimate} cannot be replaced by any uniform power-type estimate. To be precise, let $n \geq 2$ and $\tau > p>n-1$. For every $N>1$, we construct $\lambda_N \in C(B_1)$, $u_N \in C(B_1)$, and $f_N \equiv 0$ such that
\begin{equation*}
    \|u_N\|_{L^{\infty}(B_1)} \leq 1,   \quad \|\lambda_N^{-1}\|_{L^p(B_1)} \leq C,
\end{equation*}
and
\begin{equation*}
    \mathcal{M}_{\lambda_N, 1}^-(D^2u_N) \leq 0 \quad \text{in $B_1$}
\end{equation*}
in the viscosity sense, while 
\begin{equation*}
\int_{B_{1/2}}\log^\tau(1+\|D^2u_N\|)\,\mathrm{d}x >N.
\end{equation*}
In the uniformly elliptic setting, the example in \cite[Remark~3.3]{ASS12} utilizes a radial power function to show that an admissible exponent $\varepsilon$ in the $W^{2, \varepsilon}$ estimate must converge to zero when the ellipticity ratio tends to infinity. Since our assumptions impose no pointwise upper bound on $\lambda^{-1}$, this observation already suggests that no uniform power-type estimate should be expected. In fact, we replace the power function in \cite{ASS12} by the exponential function of the form $\exp((R/r)^{\beta})$, which gives arbitrarily large $\|D^2u\|$ while keeping $\|\lambda^{-1}\|_{L^p(B_1)}$ uniformly bounded. 
A similar idea was introduced in \cite[Proposition~2.1]{Bow26}, where the failure of power-type weak Harnack estimates was demonstrated. These parallel phenomena imply that logarithmic scales arise naturally in the non-uniformly elliptic setting.

When $n=2$, the exponent in Theorem~\ref{thm-log-hessian-estimate} equals $p$, and hence the example proves its sharpness. In dimensions $n \geq 3$, the optimal exponent remains open in the range
\begin{equation*}
    \frac{p}{n-1}<\tau\leq p.
\end{equation*}

\subsection{Strong solutions}
The continuity and positivity of $\lambda$ imply uniform ellipticity on each compact subset of $B_1$. However, the estimates in this paper are independent of any pointwise lower bound for $\lambda$. It is therefore natural to ask whether the estimate continues to hold when $\lambda$ is merely measurable and may vanish on a set of measure zero. In Section~\ref{sec:strong-solution}, we establish the logarithmic Hessian estimate for $L^n$-strong supersolutions under the same integrability assumptions on $\lambda^{-1}$ and $f^+/\lambda$; see Theorem~\ref{thm-log-hessian-estimate-strong}. The global approach is useful here, since taking the minimum of a strong supersolution and a quadratic barrier need not preserve $W^{2,n}$ regularity. The key ingredient is the measure estimate in Lemma~\ref{lem-nonhomo-measure-est-strong}, which we prove by smooth approximation in $W^{2,n}$, instead of the inf-convolution argument used in Lemma~\ref{lem-nonhomo-measure-est}.

We also extend the logarithmic local maximum principle of \cite[Proposition~1.4]{Bow26} to $L^n$-strong subsolutions; see Proposition~\ref{prop-log-local-maximum-strong}. The strong solution argument in \cite[Theorem~6.1]{Bow26} assumes $s>n$ to justify the application of the area formula to the associated $W^{1,s}$ maps, which satisfy Lusin's $N$-property. The endpoint $s=n$ requires additional care, since $W^{1,n}$ regularity alone does not justify the direct application of the area formula. We overcome this difficulty in Lemmas~\ref{lem-nonhomo-measure-est-strong} and ~\ref{lem-log-measure-strong} by applying the area formula to smooth approximations and passing to the limit using the stability of the contact sets and the $L^1$ convergence of the Jacobians.

\subsection{Grushin-type operators}
For $x=(y,z)\in\mathbb{R}^{n_1}\times\mathbb{R}^{n_2}$ with $n_1,n_2\geq1$, $n=n_1+n_2$, and $\gamma\in\mathbb{R}$, we consider the \emph{Grushin-type operator}
\begin{equation*}
    G_{\gamma}u\coloneqq |y|^{\gamma}\Delta_yu+\Delta_zu.
\end{equation*}
When $\gamma<0$, the equation $G_{\gamma}u=0$ is understood in the normalized form
\begin{equation*}
    \Delta_yu+|y|^{-\gamma}\Delta_zu=0.
\end{equation*}
Grushin-type operators, including the classical model studied by Grushin~\cite{Gru70}, provide basic models of non-uniform ellipticity, with degeneracy or singularity on the set $\{y=0\}$ when $\gamma\neq0$. For regularity results for Grushin-type equations, we refer to \cite{Mal18, Mon14} for Harnack inequalities, and to \cite{KLY24, STV21a, STV21b} for Schauder estimates and asymptotic expansions in related codimension-one problems ($n_1=1$).

As an application of our logarithmic Hessian estimate, we derive local regularity for bounded $L^n$-strong solutions of $G_{\gamma}u=0$ in $B_1$ when $|\gamma|<n_1/(n(n-1))$. In this range, solutions are smooth in the $z$-variables and belong locally to $C^{2,-\gamma}$ when $\gamma<0$, and to $C^{1,\beta}$ for every $\beta\in(0,1-\gamma)$ when $\gamma>0$; see Proposition~\ref{prop-Grushin-strong}. These H\"older exponents are sharp for $\gamma<0$ and nearly sharp for $\gamma>0$, as illustrated by an explicit example in Remark~\ref{rmk-Grushin-exponent}.

Our argument does not rely on an intrinsic metric or a weighted reformulation specific to the Grushin operator. A key observation is that the coefficients of $G_{\gamma}$ are independent of the $z$-variables, and thus difference quotients in the $z$-directions still satisfy the same equation. Combining this translation invariance with Theorem~\ref{thm-log-hessian-estimate-strong} and Proposition~\ref{prop-log-local-maximum-strong}, we obtain local bounds for the second derivatives in the $z$-variables. A bootstrap argument then gives bounds for derivatives of every order in the $z$-directions. Regularity in the remaining directions follows from classical estimates for the Poisson equation.

The same regularity also holds for locally uniform limits of viscosity solutions of the regularized equations considered in Remark~\ref{rmk-Grushin-viscosity}, although it need not hold for arbitrary $C$-viscosity solutions of the Grushin equation. Moreover, the argument in the invariant directions applies to linear equations whose coefficients are independent of those variables, under suitable integrability assumptions on the ellipticity; see Remark~\ref{rmk-invariant-directions}. Similar directional estimates in the uniformly elliptic setting can be found in the partial Schauder theory of Dong--Kim~\cite{DK11, DK19}.

\subsection{Organization of the paper}
The paper is organized as follows. In Section~\ref{sec-preliminaries}, we introduce the viscosity framework, convex envelopes, and contact sets, and prove the measure estimate underlying Theorem~\ref{thm-log-hessian-estimate}. In Section~\ref{sec-main}, we prove Theorem~\ref{thm-log-hessian-estimate} and Corollary~\ref{cor-log-hessian}. Section~\ref{sec-counterexample} is devoted to the construction of the counterexample described above. In Sections~\ref{sec:strong-solution} and ~\ref{sec-application}, we extend the logarithmic Hessian estimate to strong supersolutions and apply it to study the regularity of solutions of Grushin-type equations.

\section*{Acknowledgments}
The authors thank David Bowman for bringing his recent work~\cite{Bow26} to their attention and for helpful discussions. 

S.-C. Lee was supported by the KIAS Individual Grant (No. MG099001) at the Korea Institute for Advanced Study.

\section{Preliminaries}\label{sec-preliminaries}
In this section, we introduce several notions such as viscosity solutions, convex envelopes and contact sets that will be used throughout the paper, and establish the measure estimate underlying the proof of Theorem~\ref{thm-log-hessian-estimate}. Throughout the paper, we assume that $n\geq 2$, and let $\Omega$ be a bounded convex domain in $\mathbb{R}^n$. For matrices, $\|\cdot\|$ denotes the Frobenius norm.

\begin{definition}[Elliptic fully nonlinear operator]
  Let $\lambda \in C(\Omega)$ with $0<\lambda(\cdot) \leq 1$. We say that a continuous function $F : \mathcal{S}^n \times \Omega \to \mathbb{R}$ is \emph{$(\lambda(\cdot),1)$-elliptic} if for any $M, N \in \mathcal{S}^n$ with $N \geq 0$ and $x \in \Omega$
  \begin{equation*}
      \lambda(x)\tr N \leq F(M+N, x)-F(M, x) \leq \tr N.
  \end{equation*}
\end{definition}
Pucci extremal operators $\mathcal{M}_{\lambda, 1}^{\pm}$ presented in Section~\ref{sec-introduction} are typical examples of $(\lambda(\cdot), 1)$-elliptic fully nonlinear operators.

\begin{definition}[$C$-viscosity solution]
   Let $F\in C(\mathcal{S}^n\times\Omega)$ and $f\in C(\Omega)$. A function $u \in C(\Omega)$ is a \emph{$C$-viscosity subsolution} of \(  F(D ^2 u(x), x) \geq f(x)\) in $\Omega$, when the following condition holds: if \(  x _0 \in \Omega  ,  \; \varphi \in C ^2 (\Omega) 
   \) and \(  u-\varphi  \) has a local maximum at \(  x _0   \) then 
   \begin{equation*}
      F(D ^2\varphi (x_0), x_0) \geq f(x_0).
   \end{equation*}
   In a similar way, a function $u \in C(\Omega)$ is a \emph{$C$-viscosity supersolution} of \(  F(D ^2 u(x), x) \leq f(x)\) in $\Omega$, when the following condition holds: if \(  x _0 \in \Omega  ,  \; \varphi \in C ^2 (\Omega) 
   \) and \(  u-\varphi  \) has a local minimum at \(  x _0   \) then 
   \begin{equation*}
      F(D ^2\varphi (x_0), x_0) \leq f(x_0).
   \end{equation*}
   Finally, $u$ is a \emph{$C$-viscosity solution} of $F(D^2u,x)=f$ in $\Omega$ if it is both a $C$-viscosity subsolution and a $C$-viscosity supersolution.
 \end{definition}

We recall several notions in convex analysis from \cite[Section~3]{Moo19}. For $\xi \in \mathbb{R}^n$ and $a, b \in \mathbb{R}$, we define paraboloids $P^a_{\xi, b}$ by
\begin{equation*}
    P^a_{\xi, b}(x) \coloneqq \frac{a}{2}|x|^2+\xi\cdot x+b.
\end{equation*}
We say the paraboloid $P^a_{\xi, b}$ has opening $a$. For a function $v \in C(\Omega)$, we define the \emph{$a$-convex envelope} $\Gamma_v^a$ on $\Omega$ by 
\begin{equation*}
    \Gamma_v^a(x; \Omega)=\Gamma_v^a(x) \coloneqq \sup_{\xi \in \mathbb{R}^n, \, b\in \mathbb{R}} \left\{P^{-a}_{\xi, b}(x) : P^{-a}_{\xi, b} \leq v \ \text{in} \ \Omega \right\}.
\end{equation*}
Here the supremum of the empty set is understood to be $-\infty$, and $\Gamma_v^0$ is the usual convex envelope. Moreover, we define the \emph{lower $a$-contact set} by
\begin{equation*}
    A_a(v; \Omega)=A_a(v) \coloneqq \{x \in \Omega : v(x)=\Gamma_v^a(x)\}
\end{equation*}
to be those points in $\Omega$ where $v$ has a tangent paraboloid of opening $-a$ from below in $\Omega$. The contact sets are monotone in the opening, i.e.,
\begin{equation*}
    A_a(v) \subset A_{a'}(v) \quad \text{whenever $0 \leq a \leq a'$}.
\end{equation*}
We also define
 \begin{align*}
    \underline{\Theta}_u(x; \Omega) &= \inf\{a \geq 0: x\in A_a(u; \Omega)\},\\
    \overline{\Theta}_u(x; \Omega) &=  \inf\{a \geq 0: x\in A_a(-u; \Omega)\},   \\
    \Theta_u(x; \Omega)&=\max\{\overline{\Theta}_u(x; \Omega),\underline{\Theta}_u(x; \Omega)\}.
\end{align*}
Here the infimum of the empty set is understood to be $+\infty$. We provide an equivalent definition for $\underline{\Theta}_u$ as follows (an analogous characterization holds for $\overline{\Theta}_u$). 
\begin{equation*}
    \underline{\Theta}_u(x; \Omega)
\coloneqq
\inf\left\{ a\geq 0 :
\begin{array}{l}
\text{there exists } \xi\in\mathbb{R}^n \text{ such that}\\[1mm]
u(y)\geq u(x)+\xi\cdot(y-x)-\dfrac{a}{2}|y-x|^2
\quad\text{for every }y\in \Omega
\end{array}
\right\}.
\end{equation*}

We now repeat the arguments in \cite[Section~3]{Moo19} with suitable modifications. For functions on $B_1$, we write $\Gamma_v^a=\Gamma_v^a(\cdot; B_1)$ and $A_a(v)=A_a(v; B_1)$.
\begin{lemma}\label{lem-sliding}
    Let $v\in C(B_1)$ be a $C$-viscosity supersolution of
    \begin{equation*}
        \mathcal{M}_{\lambda,1}^-(D^2v)\leq f \quad \text{in $B_1$},
    \end{equation*}
    where $f,\lambda\in C(B_1)$ and $0<\lambda\leq 1$. For $a>0$ and a measurable set $F\subset B_1\setminus A_0(v)$, take the tangent paraboloids of opening $-a$ to $\Gamma_v^0$ on $F$, and slide them up until they touch $v$ on a set $E$.
    Then 
    \begin{equation*}
        E\cap A_0(v)=\varnothing.
    \end{equation*}
\end{lemma}

\begin{proof}
The proof is essentially the same as that of \cite[Lemma~3.1]{Moo19}. Suppose that $x_1\in E\cap A_0(v)$. Since $\lambda(x_1)>0$ and $f(x_1)$ is finite, the same test function argument leads to a contradiction.
\end{proof}

\begin{lemma}\label{lem-mooney}
    Let $v\in C(B_1)$ be convex. For $a>0$ and a measurable set $F\subset B_1$, let $V$ denote the set of vertices of all paraboloids of opening $-a$ touching $v$ from below in $B_1$ at points of $F$. Then $|F|\leq |V|$.
\end{lemma}

\begin{proof}
    The proof is the same as that of \cite[Lemma~3.2]{Moo19}.
\end{proof}

We finally prove the measure estimate, which is a key ingredient of the proof of Theorem~\ref{thm-log-hessian-estimate}.
\begin{lemma}[Measure estimate for viscosity supersolutions]
\label{lem-nonhomo-measure-est}
    Let $u\in C(\overline{B}_1)$ be a $C$-viscosity supersolution of 
    \begin{equation*}
        \mathcal{M}_{\lambda,1}^-(D^2u)\leq f
        \quad \text{in $B_1$},
    \end{equation*}
    where $f,\lambda\in C(B_1)$ and $0<\lambda\leq 1$.
    For $a>0$ and a measurable set $F\subset B_1 \setminus A_a(u)$, take the tangent paraboloids of opening $-2a$ to $\Gamma_u^a$ on $F$, and slide them up until they touch $u$ on a set $E$. If $E\Subset B_1$, then there exists a constant $C>0$, depending only on $n$, such that
    \begin{equation*}
        |F|\leq C\int_{A_{2a}(u)\setminus A_a(u)}
        \left[\lambda^{-(n-1)}+\lambda^{-(n-1)}\frac{f^+}{a\lambda}+\left(\frac{f^+}{a\lambda}\right)^n\right]\,\mathrm{d}x.
    \end{equation*}
\end{lemma}

\begin{proof}
   We essentially follow the proof of \cite[Lemma~3.3]{Moo19}, with some modifications to deal with the non-uniform ellipticity and the right-hand side. By inner approximation, it suffices to prove the estimate when $F$ is compact. Let 
    \begin{equation*}
        v=\frac{1}{a}u+\frac{1}{2}|x|^2.
    \end{equation*}
    Then the arguments in \cite[Section~3]{Moo19} show that $A_a(u)=A_0(v)$, $A_{2a}(u)=A_1(v)$, and that the set $E$ is obtained by taking the paraboloids of opening $-1$ which are tangent from below to $\Gamma_v^0$ on $F$, and sliding them up until they touch $v$. Let $V$ denote the set of vertices of these paraboloids. We remark that both $V$ and $E$ are compact.

    Since
    \begin{equation}\label{nonhomo-eq-pucci}
        \mathcal{M}_{\lambda, 1}^-(D^2v) \leq \mathcal{M}_{\lambda, 1}^+(I)+\frac{1}{a}\mathcal{M}_{\lambda, 1}^-(D^2u) \leq n+\frac{f^+}{a},
    \end{equation}
    we apply Lemma~\ref{lem-sliding} for $v$ to find that $E  \cap A_0(v)=\varnothing$. The condition $E \Subset B_1$ and the definition of $E$ yield that 
    \begin{equation*}
        E  \subset A_1(v) \setminus A_0(v)=A_{2a}(u) \setminus A_a(u).
    \end{equation*}
    
    Assume for the moment that $u$ is semi-concave. Let $x\in E$ with corresponding vertex $y\in V$. Then the function $z\mapsto v(z)+|z-y|^2/2$ attains its minimum on $\overline{B}_1$ at $x$. Since $v$ is semi-concave and has a supporting paraboloid from below at $x$, $v$ is differentiable at $x$ and so $y=x+\nabla v(x)$. By Alexandrov's theorem, $v$ is twice differentiable a.e., and the minimum property gives $I+D^2v(x)\geq 0$ a.e. on $E$. Since the map $x \mapsto y=x+\nabla v(x)$ is Lipschitz on $E$ and maps $E$ onto $V$, the area formula gives
    \begin{equation*}
      |V| \leq \int_E \det (I+D^2v)\,\mathrm{d}x.
    \end{equation*}
    We next provide the eigenvalue estimates. At a point $x \in E$ where $v$ is twice differentiable, \eqref{nonhomo-eq-pucci} yields
\begin{equation*}
    \mathcal{M}_{\lambda,1}^-(D^2v(x))
    \leq n+\frac{f^+(x)}{a}.
\end{equation*}
Since every negative eigenvalue of $D^2v(x)$ is at least $-1$,
any positive eigenvalue $\mu$ satisfies
\begin{equation*}
    \lambda(x)\mu-(n-1)
    \leq \mathcal{M}_{\lambda,1}^-(D^2v(x))
    \leq n+\frac{f^+(x)}{a},
\end{equation*}
which gives
\begin{equation*}
    -I\leq D^2v(x)
    \leq \frac{2n-1+f^+(x)/a}{\lambda(x)}I.
\end{equation*}
Note also that the smallest eigenvalue of $D^2u(x)$ is at most $f^{+}(x)/(n\lambda(x))$, since $\mathcal{M}^-_{\lambda, 1}(D^2u)\leq f^{+}$. By scaling back, the smallest eigenvalue of $D^2v(x)$ is at most $1+f^+(x)/(an\lambda(x))$. Therefore, we conclude that there exists a constant $C>0$, depending only on $n$, such that a.e. on $E$, 
    \begin{equation*}
        \begin{aligned}
            \det D_xy=\det (I+D^2v) &\leq \left(2+  \frac{f^+}{an\lambda}\right)\left(1+ \frac{(2n-1) + f^+/a  }{\lambda} \right)^{n-1} \\
            &\leq C \left(1+\frac{f^+}{a\lambda}\right)\left(\frac{1}{\lambda}+ \frac{f^+}{a\lambda} \right)^{n-1}\\
            &\leq C\left[\lambda^{-(n-1)}+\lambda^{-(n-1)}\frac{f^+}{a \lambda}+\left(\frac{f^+}{a\lambda}\right)^{n}\right],
        \end{aligned}
    \end{equation*}
    and so
    \begin{equation*}
        |V| \leq C\int_E \left[\lambda^{-(n-1)}+\lambda^{-(n-1)}\frac{f^+}{a \lambda}+\left(\frac{f^+}{a\lambda}\right)^{n}\right] \,\mathrm{d}x.
    \end{equation*}
    By using Lemma~\ref{lem-mooney} and the inclusion $E \subset A_{2a}(u) \setminus A_a(u)$, we obtain the desired estimate when $u$ is semi-concave.
    \medskip
    
We now remove the semi-concavity assumption by employing the standard inf-convolution argument. Fix $F$ and the vertex set $V$ from $u$. For $k \geq 1$, we define the inf-convolution
\begin{equation*}
    u_k(x) \coloneqq \inf_{y \in \overline B_1} \left\{ u(y)+\frac{k}{2}|x-y|^2\right\} \quad \text{for $x \in \overline B_1$}.
\end{equation*}
Then $u_k$ is semi-concave and $u_k \to u$ uniformly on $\overline{B}_1$ as $k \to \infty$. Moreover, for $R_k \coloneqq (2\osc_{B_1}u/k)^{1/2}$, \cite[Lemma~A.3]{CCKS96} implies that $u_k$ satisfies
\begin{equation*}
    \mathcal{M}_{\lambda_k, 1}^-(D^2u_k) \leq f_k \quad \text{in $B_{1-R_k}$},
\end{equation*}
where
\begin{equation*}
    \lambda_k(x) \coloneqq \min_{\overline{B}_{R_k}(x)}\lambda \quad \text{and} \quad f_k(x) \coloneqq \max_{\overline{B}_{R_k}(x)}f^+.
\end{equation*}
We let $v_k=u_k/a+|x|^2/2$ and slide the paraboloids of opening $-1$ with vertices in $V$ from below $v_k$ until they touch on a set $E_k$.
For every sufficiently small $\delta>0$, let
\begin{equation*}
    E^\delta\coloneqq
    \{x\in B_1:\dist(x,E)<\delta\}\Subset B_1.
\end{equation*}
Uniform convergence and compactness imply that
$E_k\subset E^\delta$ for all sufficiently large $k$;
see also \cite[Lemma~A.1]{CCKS96}.
Since $R_k\to 0$, we also have $E^\delta\Subset B_{1-R_k}$
for all sufficiently large $k$.
Applying the preceding semi-concave argument to $u_k$, we obtain
\begin{equation*}
    |V|\leq C\int_{E_k}
    \left[
        \lambda_k^{-(n-1)}
        +\lambda_k^{-(n-1)}\frac{f_k}{a\lambda_k}
        +\left(\frac{f_k}{a\lambda_k}\right)^n
    \right]\,\mathrm{d}x.
\end{equation*}
Since $\lambda_k\to\lambda$ and $f_k\to f^+$ uniformly on
$\overline{E^\delta}$ as $k \to\infty$, we have
\begin{equation*}
    \begin{aligned}
         |V| &\leq C\limsup_{k \to \infty}\int_{E^{\delta}}\left[\lambda_k^{-(n-1)}+\lambda_k^{-(n-1)}\frac{f_k}{a \lambda_k}+\left(\frac{f_k}{a\lambda_k}\right)^{n}\right]\,\mathrm{d}x \\
         &=C\int_{E^{\delta}} \left[\lambda^{-(n-1)}+\lambda^{-(n-1)}\frac{f^+}{a \lambda}+\left(\frac{f^+}{a\lambda}\right)^{n}\right]\,\mathrm{d}x.
    \end{aligned}
\end{equation*}
Letting $\delta\to 0^+$ by the dominated convergence theorem
and arguing as in the semi-concave case, we finish the proof.
\end{proof}
  
\section{Logarithmic Hessian estimates}\label{sec-main}
In this section, we prove Theorem~\ref{thm-log-hessian-estimate} and Corollary~\ref{cor-log-hessian}. We essentially employ the argument of \cite[Section~5]{Moo19} to keep the new contact points inside the domain. We first assume that $u\in C(\overline{B}_1)$, and then pass to the limit on smaller balls to remove this assumption.

\begin{proof}[Proof of Theorem~\ref{thm-log-hessian-estimate}]
(\emph{Step 1}. $u\in C(\overline{B}_1)$.) We first prove the estimate under the normalized assumptions
\begin{equation}\label{eq-global-normalization}
    0\leq u\leq 2^{-8n} \quad \text{and} \quad \left\|\frac{f^+}{\lambda}\right\|_{L^{\sigma}(B_1)}\leq 1.
\end{equation}
Throughout Step 1, we write $A_a(u)=A_a(u;B_1)$. For $k\geq 0$, we set
\begin{equation*}
    D_k\coloneqq A_{2^{k+1}}(u)\setminus A_{2^k}(u), \quad \rho_k\coloneqq 2^{-4n-k/2}, \quad F_k\coloneqq B_{1-4\rho_k}\setminus A_{2^k}(u).
\end{equation*}
We take the paraboloids of opening $-2^{k+1}$ which are tangent from below to $\Gamma_u^{2^k}$ on $F_k$, and slide them up until they touch $u$ on a set $E_k$. We claim that
\begin{equation}\label{eq-global-interior-contact}
    E_k\subset B_{1-2\rho_k}\Subset B_1.
\end{equation}
Indeed, we have
\begin{equation*}
    0\leq\Gamma_u^{2^k}\leq u\leq 2^{-8n}=2^k\rho_k^2.
\end{equation*}
Let $P$ be one of the above paraboloids touching $\Gamma_u^{2^k}$ from below at $x_0\in F_k$, and let $y$ be its vertex, i.e., $P(x)=P(y)-2^k|x-y|^2$. If $y\neq x_0$, then take
\begin{equation*}
    z=x_0+\rho_k\frac{y-x_0}{|y-x_0|}\in B_1.
\end{equation*}
Since $P(x_0)=\Gamma_u^{2^k}(x_0)\geq 0$ and $P\leq\Gamma_u^{2^k}\leq 2^k\rho_k^2$, we obtain
\begin{equation*}
    2^k\big(2\rho_k|y-x_0|-\rho_k^2\big)
    =P(z)-P(x_0)\leq 2^k\rho_k^2.
\end{equation*}
Thus, it turns out that $|y-x_0|\leq\rho_k$, which also holds when $y=x_0$, and hence $y\in B_{1-3\rho_k}$. If $P+t$ touches $u$ at $x_1\in\overline{B}_1$, then
\begin{equation*}
    0\leq u(x_1)=P(x_1)+t=P(y)+t-2^k|x_1-y|^2
    \leq 2^k\rho_k^2-2^k|x_1-y|^2.
\end{equation*}
It follows that $|x_1-y|\leq\rho_k$, which finishes the proof of \eqref{eq-global-interior-contact}.

We now apply Lemma~\ref{lem-nonhomo-measure-est} with $a=2^k$, $F=F_k$ and $E=E_k$ to find that
\begin{equation}\label{eq-global-measure}
    \begin{aligned}
        |F_k|&\leq C\int_{D_k}\lambda^{-(n-1)}\,\mathrm{d}x
        +C2^{-k}\int_{D_k}\lambda^{-(n-1)}\frac{f^+}{\lambda}\,\mathrm{d}x
        +C2^{-nk}\int_{D_k}\left(\frac{f^+}{\lambda}\right)^n\,\mathrm{d}x\\
        &\leq C\int_{D_k}\lambda^{-(n-1)}\,\mathrm{d}x
        +C2^{-k}(1+\|\lambda^{-1}\|_{L^p(B_1)}^p).
    \end{aligned}
\end{equation}
Here we used Young's inequality with exponents $p/(n-1)$ and $p/(p-n+1)$, together with \eqref{eq-global-normalization} and the choice of $\sigma$.

There are two cases to consider. If $|F_k|\geq |B_1 \setminus A_{2^k}(u)|/2$, then \eqref{eq-global-measure} directly controls $|B_1 \setminus A_{2^k}(u)|$. Otherwise, we have
\begin{equation*}
    |B_1 \setminus A_{2^k}(u)|\leq 2|(B_1 \setminus A_{2^k}(u))\setminus F_k|
    \leq 2|B_1\setminus B_{1-4\rho_k}|\leq C(n)2^{-k/2}.
\end{equation*}
Therefore, in either case,
\begin{equation}\label{eq-global-recurrence}
    |B_1 \setminus A_{2^k}(u)|\leq C\int_{D_k}\lambda^{-(n-1)}\,\mathrm{d}x
    +C2^{-k}(1+\|\lambda^{-1}\|_{L^p(B_1)}^p)+C2^{-k/2}.
\end{equation}
For
\begin{equation*}
    \tau\coloneqq\frac{p}{n-1}>1 \quad \text{and} \quad
    S_N\coloneqq\sum_{k=0}^N(k+1)^{\tau-1}|B_1 \setminus A_{2^k}(u)|,
\end{equation*}
we use H\"older's inequality and the disjointness of $D_k$ to obtain that
\begin{equation}\label{eq-global-summation}
    \begin{aligned}
        S_N&\leq C\sum_{k=0}^N(k+1)^{\tau-1}\int_{D_k}\lambda^{-(n-1)}\,\mathrm{d}x
        +C(1+\|\lambda^{-1}\|_{L^p(B_1)}^p)\\
        &\leq C\|\lambda^{-1}\|_{L^p(B_1)}^{n-1}
        \left(\sum_{k=0}^N(k+1)^\tau|D_k|\right)^{(\tau-1)/\tau}
        +C(1+\|\lambda^{-1}\|_{L^p(B_1)}^p).
    \end{aligned}
\end{equation}
Moreover, by rearranging the sum, we obtain
\begin{equation*}
    \begin{aligned}
        \sum_{k=0}^N(k+1)^\tau|D_k|
        &=|B_1 \setminus A_{1}(u)|+\sum_{k=1}^N\big((k+1)^\tau-k^\tau\big)|B_1 \setminus A_{2^k}(u)|-(N+1)^\tau|B_1 \setminus A_{2^{N+1}}(u)|\\
        &\leq \tau\sum_{k=0}^N(k+1)^{\tau-1}|B_1 \setminus A_{2^k}(u)|=\tau S_N.
    \end{aligned}
\end{equation*}
Combining this with \eqref{eq-global-summation} and using Young's inequality, we conclude that
\begin{equation*}
    S_N\leq \frac{1}{2}S_N+C(1+\|\lambda^{-1}\|_{L^p(B_1)}^p).
\end{equation*}
In particular, by letting $N\to\infty$, we obtain
\begin{equation*}
    \sum_{k=0}^{\infty}(k+1)^{\tau-1}|B_1 \setminus A_{2^k}(u)|
    \leq C(1+\|\lambda^{-1}\|_{L^p(B_1)}^p).
\end{equation*}
Since $\{\underline{\Theta}_u(\cdot;B_1)>2^k\}\subset B_1 \setminus A_{2^k}(u)$, we apply Tonelli's theorem, a dyadic decomposition, and the mean value theorem to find that
\begin{equation}\label{eq-global-normalized-estimate}
    \begin{aligned}
        \int_{B_1}\log^\tau(1+\underline{\Theta}_u(x;B_1))\,\mathrm{d}x
        &=\tau\int_0^{\infty}\frac{\log^{\tau-1}(1+t)}{1+t}
        |\{x\in B_1:\underline{\Theta}_u(x;B_1)>t\}|\,\mathrm{d}t\\
        &\leq C|B_1|+\tau\sum_{k=0}^{\infty}
        \int_{2^k}^{2^{k+1}}\frac{\log^{\tau-1}(1+t)}{1+t}
        |B_1\setminus A_{2^k}(u)|\,\mathrm{d}t\\
        &\leq C|B_1|+C\sum_{k=0}^{\infty}(k+1)^{\tau-1}
        |B_1\setminus A_{2^k}(u)|\\
        &\leq C(1+\|\lambda^{-1}\|_{L^p(B_1)}^p).
    \end{aligned}
\end{equation}
This proves the estimate under the normalized assumptions.

For the general case, we set
\begin{equation*}
    K\coloneqq 1+\|u\|_{L^{\infty}(B_1)}+\left\|\frac{f^+}{\lambda}\right\|_{L^{\sigma}(B_1)}
\end{equation*}
and define
\begin{equation*}
    v=\frac{2^{-8n-1}}{K}\left(u+\|u\|_{L^{\infty}(B_1)}\right).
\end{equation*}
Then $v$ satisfies $\mathcal{M}_{\lambda, 1}^-(D^2v) \leq \tilde f=2^{-8n-1}f/K$. In particular, $v$ satisfies the normalization \eqref{eq-global-normalization}, and
\begin{equation*}
    \underline{\Theta}_v(x;B_1)=2^{-8n-1}\,\underline{\Theta}_u(x;B_1)/K.
\end{equation*}
Since $2^{8n+1}K >1$, we have
\begin{equation*}
    \log(1+t)\leq\log(2^{8n+1}K)+\log(1+2^{-8n-1}t/K) \quad \text{for $t\geq 0$}.
\end{equation*}
Applying \eqref{eq-global-normalized-estimate} to $v$ proves the estimate when $u\in C(\overline{B}_1)$.

\medskip
(Step 2. $u\in C(B_1)\cap L^{\infty}(B_1)$.) For $r \in (1/2, 1)$, let $u_r(x)=u(rx)$, $\lambda_r(x)=\lambda(rx)$ and $f_r(x)=r^2f(rx)$. Then $u_r\in C(\overline{B}_1)$ satisfies $\mathcal{M}_{\lambda_r,1}^-(D^2u_r)\leq f_r$ in $B_1$, and
\begin{equation*}
    \underline{\Theta}_{u_r}(x;B_1)
    =r^2\underline{\Theta}_u(rx;B_r).
\end{equation*}
Applying Step 1 to $u_r$ and scaling back, we obtain
\begin{equation*}
    \begin{aligned}
        \int_{B_r}\log^\tau(1+r^2\underline{\Theta}_u(x;B_r))\,\mathrm{d}x
        &\leq C(1+\|\lambda^{-1}\|_{L^p(B_1)}^p)\\
        &\qquad+C\log^\tau\left(1+\|u\|_{L^{\infty}(B_1)}
        +\left\|\frac{f^+}{\lambda}\right\|_{L^{\sigma}(B_1)}\right),
    \end{aligned}
\end{equation*}
where $C$ is independent of $r$. Letting $r \to 1^-$, we obtain the desired estimate by the monotone convergence theorem, since $\underline{\Theta}_u(x;B_r)\nearrow\underline{\Theta}_u(x;B_1)$ for every $x\in B_1$.
\end{proof}

\begin{remark}[Interior and global estimates]\label{rmk-interior-global}
We point out that Mooney~\cite[Theorems~2.1 and~2.4]{Moo19} proves the interior and global estimates separately. In his notation, with ellipticity ratio $\Lambda\geq 1$, the respective exponents in the measure decay estimates are
\begin{equation*}
    \varepsilon_{\mathrm{int}}=2^{-2}(2n\Lambda)^{1-n}
    \quad \text{and} \quad
    \varepsilon_{\mathrm{glob}}=2^{-3}(2n\Lambda)^{1-n}.
\end{equation*}
Thus, the interior argument gives a better exponent than the global one.

In our setting, the boundary layer term in \eqref{eq-global-recurrence} is summable with the weights used in \eqref{eq-global-summation}. More precisely, for $\tau=p/(n-1)$, we have
\begin{equation*}
    \sum_{k=0}^{\infty}(k+1)^{\tau-1}2^{-k/2}<\infty.
\end{equation*}
Consequently, this term only contributes to the constant and does not reduce the logarithmic exponent $p/(n-1)$. The global estimate therefore yields the interior estimate with the same exponent by restricting the integral to $B_{1/2}$. For this reason, a separate interior argument is not necessary for Theorem~\ref{thm-log-hessian-estimate}.
\end{remark}

\begin{proof}[Proof of Corollary~\ref{cor-log-hessian}]
    By applying Theorem~\ref{thm-log-hessian-estimate} to $u$ and $-u$, we obtain the estimates for $\underline{\Theta}_u(\cdot;B_1)$ and $\underline{\Theta}_{-u}(\cdot;B_1)=\overline{\Theta}_u(\cdot;B_1)$. In particular, $\Theta_u=\max\{\overline{\Theta}_u,\underline{\Theta}_u\}<\infty$ a.e. in $B_1$. Thus, the generalized Rademacher--Stepanov theorem (see \cite{CZ61} for instance) implies that $u$ is twice differentiable a.e. in $B_1$, and
    \begin{equation*}
        -\underline{\Theta}_u(x;B_1)I\leq D^2u(x)\leq\overline{\Theta}_u(x;B_1)I.
    \end{equation*}
    The desired Hessian estimate on $B_1$ follows by combining the two contact estimates.
\end{proof}

\section{A counterexample}\label{sec-counterexample}
\begin{proposition}
Let $n\geq 2$, $p>n-1$, fix $\tau>p$, and set $\sigma=\max\{p/(p-n+1), n\}$. Then no uniform logarithmic Hessian estimate with exponent $\tau$ can hold under the assumptions $\lambda^{-1} \in L^p(B_1)$ and $f/\lambda \in L^{\sigma}(B_1)$. More precisely, for every $N>1$, there exist functions
\begin{itemize}
    \item $\lambda_N \in C(B_1)$ with $0<\lambda_N \leq 1$,
    \item $u_N \in C(B_1)$,
    \item $f_N \equiv 0$,
\end{itemize}
such that
\begin{equation*}
    \|u_N\|_{L^{\infty}(B_1)} \leq 1, \quad \|\lambda_N^{-1}\|_{L^p(B_1)} \leq C, \quad \|f_N/\lambda_N\|_{L^{\sigma}(B_1)}=0,
\end{equation*}
where $C>0$ depends only on $n$, $p$ and $\tau$, and $u_N$ is a viscosity supersolution of
\begin{equation*}
    \mathcal{M}_{\lambda_N, 1}^{-}(D^2u_N) \leq 0 \quad \text{in $B_1$},
\end{equation*}
while
\begin{equation*}
   \int_{B_{1/2}}\log^\tau(1+\underline{\Theta}_{u_N}(x; B_1))\,\mathrm{d}x >N.
\end{equation*}
The same construction also satisfies
\begin{equation*}
    \int_{B_{1/2}}\log^\tau(1+\|D^2u_N(x)\|)\,\mathrm{d}x >N.
\end{equation*}
\end{proposition}

\begin{proof}
We set $\beta=1/p+1/\tau$ so that $\beta p<2$ and $\beta \tau>2$. For $R>0$ to be determined later, we define
\begin{equation*}
    \psi_R(r)=
    \begin{cases}
        R^2e^{(R/r)^{\beta}}+\dfrac{\beta}{2}er^2-\left(1+\dfrac{\beta}{2}\right)e R^2 &\text{if $0<r<R$},\\
        0 &\text{if $r \geq R$}.
    \end{cases}
\end{equation*}
Then it is immediate to check that $\psi_R(R)=\psi_R'(R)=0$ and so $\psi_R \in C^1(0, \infty)$. Moreover, a direct computation gives, for $0<r<R$,
\begin{equation*}
    \begin{aligned}
        \psi_R''(r)&=\beta\left(\frac{R}{r}\right)^{\beta+2}e^{(R/r)^{\beta}}\left[\beta\left(\frac{R}{r}\right)^{\beta}+\beta+1\right]+\beta e>0 \\
        \frac{\psi_R'(r)}{r}&=-\beta\left(\frac{R}{r}\right)^{\beta+2}e^{(R/r)^{\beta}}+\beta e<0.
    \end{aligned}
\end{equation*}
For simplicity, we set
\begin{equation*}
    \kappa_R(r) \coloneqq \left[\beta\left(\frac{R}{r}\right)^{\beta}+\beta+1\right]^{-1} \in (0,1).
\end{equation*}
Since $\psi_R'(r)<0$ in $r \in (0, R)$ and $\lim_{r \to 0^+}\psi_R(r)=\infty$, let $r_R \in (0, R)$ be the unique radius satisfying $\psi_R(r_R)=1$. Then since
\begin{equation*}
    1=R^2e^{(R/r_R)^{\beta}}+\dfrac{\beta}{2}er_R^2-\left(1+\dfrac{\beta}{2}\right)e R^2
\end{equation*}
and
\begin{equation*}
    \left|\dfrac{\beta}{2}er_R^2-\left(1+\dfrac{\beta}{2}\right)e R^2 \right| \leq C_{\beta}R^2,
\end{equation*}
we observe that 
\begin{equation*}
    |e^{(R/r_R)^{\beta}}-R^{-2}| \leq C_{\beta},
  \end{equation*}
and so
\begin{equation}\label{eq-limit-rR}
    \lim_{R \to 0}\frac{r_R}{R}=0.
\end{equation}

For $x' \in \mathbb{R}^2$, we define
\begin{equation*}
        V_R(x')
        \coloneqq   
        \sum_{y'\in A_R} \min\{1, \psi_R(|x'-2Ry'|)\},
      \end{equation*}    
where
\begin{equation*}
    A_R \coloneqq \{y' \in \mathbb{Z}^2 : B_R(2Ry') \subset B_{1/4}\} \quad \text{for $R>0$ small}.
\end{equation*}
We note that the supports of the summands are pairwise disjoint up to their boundaries, their number is comparable to $R^{-2}$, $0 \leq V_R \leq 1$, and $\mathrm{supp}V_R \subset B_{1/4}$. Moreover, we observe that
\begin{equation*}
    \mathcal{M}^-_{\kappa_R, 1}(D^2\psi_R(|x'|))=\kappa_R(r) \psi_R''(r)+\frac{\psi_R'(r)}{r}=\beta e(\kappa_R(r)+1) \quad \text{if $0<r<R$}.
\end{equation*}
Finally, we set for $x=(x', x'') \in \mathbb{R}^2 \times \mathbb{R}^{n-2}$, 
\begin{equation*}
    u_R(x) \coloneqq \left(1+\frac{\beta}{2}e\right)^{-1} \left(V_R(x')-\frac{\beta}{2}e|x'|^2\right) \eqqcolon  \left(1+\frac{\beta}{2}e\right)^{-1} W_R(x'),
\end{equation*}
and
\begin{equation*}
    \lambda_R(x) \coloneqq
    \begin{cases}
        \kappa_R(r_R) &\text{if $|x'-2Ry'| \leq r_R$ for some $y' \in A_R$},\\
        \kappa_R(|x'-2Ry'|) &\text{if $r_R < |x'-2Ry'| < R$ for some $y' \in A_R$},\\
        \kappa_R(R) &\text{otherwise}.
    \end{cases}
\end{equation*}
Here we note that $\lambda_R \in C(\mathbb{R}^n)$ with $0<\lambda_R \leq 1$ and $u_R$ is the standard normalization of $W_R$ so that $\|u_R\|_{L^{\infty}(B_1)} \leq 1$. Then we claim that
\begin{equation*}
    \mathcal{M}_{\lambda_R, 1}^-(D^2W_R)\leq 0 \quad \text{in $\mathbb{R}^n$}.
\end{equation*}
Indeed, if $r_R <r=|x'-2Ry'| < R$ for some $y' \in A_R$, then we have $\lambda_R(x)=\kappa_R(r)$, 
\begin{equation*}
    D^2_{x'x'}W_R=D_{x'x'}^2\left(R^2e^{(R/r)^{\beta}}\right)
\end{equation*}
and so 
\begin{equation*}
    \mathcal{M}_{\lambda_R, 1}^-(D^2W_R)=0.
\end{equation*}
On the other hand, if \(  r < r_R  \) or \(  r > R  \), we have $D^2_{x'x'}W_R=-\beta eI_2$ and so
\begin{equation*}
     \mathcal{M}_{\lambda_R, 1}^-(D^2W_R)=-2\beta e<0.
\end{equation*}
It only remains to verify the viscosity supersolution inequality at two interfaces $r=r_R$ and $r=R$. Suppose that a $C^2$ test function $\varphi$ touches $W_R$ from below at an interface point $x_0$ satisfying either $|x_0'-2Ry'|=r_R$ or $|x_0'-2Ry'|=R$ for some $y' \in A_R$. If we let
\begin{equation*}
    \tilde \varphi(x)=\varphi(x)+\frac{\beta}{2}e|x'|^2,
\end{equation*}
then $\tilde \varphi$ touches $V_R$ from below at $x_0$. 
\begin{enumerate}[(i)]
    \item ($r=r_R$) Since $\psi'_R(r_R)<0$ and $V_R(r)=1$ in $0<r<r_R$, it is easy to check that such a $C^2$-test function $\tilde\varphi$ cannot exist.

    \item ($r=R$) We observe that 
    \begin{equation*}
        D^2\tilde \varphi(x_0) \leq \psi_R''(R^-) \nu \otimes \nu,
    \end{equation*}
    where $\nu$ is the outward unit normal to $\partial B_R(2Ry')$. In particular, we have 
    \begin{equation*}
         \mathcal{M}_{\lambda_R(x_0), 1}^-(D^2\varphi(x_0)) \leq \kappa_R(R) \cdot \frac{\beta e}{\kappa_R(R)}-\beta e=0.
    \end{equation*}
\end{enumerate}
Therefore, we conclude that
\begin{equation*}
    \mathcal{M}_{\lambda_R, 1}^-(D^2u_R)\leq 0\eqqcolon f_R \quad \text{in $B_1$}.
\end{equation*}
We now estimate the relevant norm and integrals as follows.
\begin{equation*}
    \begin{aligned}
          \int_{B_1}\lambda_R^{-p}\,\mathrm{d}x &\leq C(n, p,\tau)+C_nR^{-2} \int_{|x'| \leq R} \kappa_R^{-p}(|x'|) \,\mathrm{d}x'\\
            &\leq C(n, p,\tau)+C_nR^{-2} R^2\int_0^1 \left[\beta s^{-\beta}+\beta+1\right]^{p}s \,\mathrm{d}s\\
            &\leq C(n, p, \tau),
    \end{aligned}
\end{equation*}
where we used that $\beta p<2$ and recalled that the number of cells is comparable to $R^{-2}$.

Furthermore, if $r_R < r=|x'-2Ry'| < R$, then we have
\begin{equation*}
    D^2W_R=
    \begin{pmatrix}
D_{x'x'}^2\left(R^2e^{(R/r)^{\beta}}\right) & 0 \\
0 & 0
\end{pmatrix}
.
\end{equation*}
Then the eigenvalues of $D^2W_R$ in this annulus are given by
\begin{equation*}
    \beta\left(\frac{R}{r}\right)^{\beta+2}e^{(R/r)^{\beta}} \cdot \frac{1}{\kappa_R(r)}, \quad -\beta\left(\frac{R}{r}\right)^{\beta+2}e^{(R/r)^{\beta}}, \quad 0, \cdots, 0,
\end{equation*}
and so we obtain
\begin{equation*}
\begin{aligned}
 \underline{\Theta}_{W_R}(x; B_1) &\geq \beta\left(\frac{R}{r}\right)^{\beta+2}e^{(R/r)^{\beta}}, \\
     \|D^2W_R\| &\geq \beta\left(\frac{R}{r}\right)^{\beta+2}e^{(R/r)^{\beta}} \cdot \frac{1}{\kappa_R(r)}.
\end{aligned}
\end{equation*}
Therefore, we conclude that
\begin{equation*}
    \begin{aligned}
         \int_{B_{1/2}} \log^\tau(1+\underline{\Theta}_{u_R}(x; B_1))\,\mathrm{d}x &\geq c(n, p, \tau) R^{-2}\int_{r_R \leq |x'| \leq R}(R/|x'|)^{\beta \tau}\,\mathrm{d}x'\\
         &\geq c(n, p, \tau) \int_{r_R/R}^1s^{-\beta \tau}s\,\mathrm{d}s.
    \end{aligned}
\end{equation*}
By recalling the asymptotic behavior \eqref{eq-limit-rR} of $r_R$ together with $\beta \tau>2$, it turns out that 
\begin{equation*}
    \int_{B_{1/2}}\log^\tau(1+\underline{\Theta}_{u_R}(x; B_1))\,\mathrm{d}x\to \infty \quad \text{as $R \to 0$}.
\end{equation*}
By choosing $R=R(N)>0$ sufficiently small, and then setting $u_N \coloneqq u_R$, $\lambda_N \coloneqq \lambda_R$, and $f_N \coloneqq 0$, we finish the desired construction.
\end{proof}

\section{Strong solutions}\label{sec:strong-solution}
In this section, we establish the global logarithmic Hessian estimate for strong supersolutions. Here $\lambda$ and $f$ are merely measurable, and $\lambda$ may vanish on a set of measure zero.

\begin{definition}[$L^{s}$-strong solution]
    Let $s\geq n$ and let $\Omega\subset\mathbb{R}^n$ be open. Assume that $F(M,x)$ is continuous in $M$ and measurable in $x$, and that $f$ is measurable. A function $u\in W^{2,s}_{\mathrm{loc}}(\Omega)$ is an \emph{$L^{s}$-strong supersolution} of $F(D^2u,x)\leq f$ in $\Omega$ if
    \begin{equation*}
        F(D^2u(x),x)\leq f(x) \quad \text{a.e. in $\Omega$}.
    \end{equation*}
    Strong subsolutions are defined with the reverse inequality, and $u$ is an \emph{$L^{s}$-strong solution} of $F(D^2u,x)=f$ if the equality holds a.e. in $\Omega$.
\end{definition}

The main result of this section is the following strong solution version of Theorem~\ref{thm-log-hessian-estimate}.
\begin{theorem}[Global estimate for strong supersolutions]\label{thm-log-hessian-estimate-strong}
    Let $f$ and $\lambda$ be measurable functions in $B_1$ such that
    \begin{equation*}
        0<\lambda\leq 1 \quad \text{a.e.}, \quad
        \lambda^{-1}\in L^p(B_1), \quad
        \frac{f^+}{\lambda}\in L^{\sigma}(B_1)
    \end{equation*}
    for some $p>n-1$ and $\sigma=\max\{p/(p-n+1), n\}$. Suppose that $u\in W^{2,n}_{\mathrm{loc}}(B_1)\cap L^{\infty}(B_1)$ is an $L^n$-strong supersolution of
    \begin{equation*}
        \mathcal{M}_{\lambda, 1}^-(D^2u)\leq f
        \quad \text{in $B_1$}.
    \end{equation*}
    Then there exists a constant $C>0$, depending only on $n$ and $p$, such that
    \begin{equation*}
        \begin{aligned}
            \int_{B_1}\log^{\frac{p}{n-1}}
            (1+\underline{\Theta}_u(x;B_1))\,\mathrm{d}x
            &\leq C(1+\|\lambda^{-1}\|_{L^p(B_1)}^p)\\
            &\quad+C\log^{\frac{p}{n-1}}
            \left(1+\|u\|_{L^{\infty}(B_1)}
            +\left\|\frac{f^+}{\lambda}\right\|_{L^{\sigma}(B_1)}\right).
        \end{aligned}
    \end{equation*}
\end{theorem}

We first prove an almost everywhere version of Lemma~\ref{lem-sliding} for Sobolev functions, without the supersolution assumption.

\begin{lemma}\label{lem-sliding-strong}
    Let $v\in W^{2,n}_{\mathrm{loc}}(B_1)$. For $a>0$ and a measurable set $F\subset B_1\setminus A_0(v)$, take the tangent paraboloids of opening $-a$ to $\Gamma_v^0$ on $F$, and slide them up until they touch $v$ on a set $E$. Then
    \begin{equation*}
        |E\cap A_0(v)|=0.
    \end{equation*}
\end{lemma}

\begin{proof}
By the Sobolev embedding, $v$ is differentiable a.e. Suppose that $x_1\in E\cap A_0(v)$ and that $v$ is differentiable at $x_1$. Let $P$ be the corresponding paraboloid touching $\Gamma_v^0$ from below at $x_0\in F$, so that $P+t$ touches
$v$ from below at $x_1$ for some $t>0$.
Let $L$ be a supporting affine function to $\Gamma_v^0$ at $x_1$. Since both $L$ and $P+t$ touch $v$ from below at $x_1$, we have $DL=Dv(x_1)=DP(x_1)$. The concavity of $P$ gives $P+t\leq L$, and hence
\begin{equation*}
    \Gamma_v^0(x_0)=P(x_0)<P(x_0)+t \leq L(x_0)\leq\Gamma_v^0(x_0),
\end{equation*}
a contradiction. Thus, $v$ is not differentiable at any point of $E\cap A_0(v)$, and the conclusion follows.
\end{proof}

We then prove the strong solution version of Lemma~\ref{lem-nonhomo-measure-est}.
Instead of the inf-convolution and semi-concavity arguments, we use smooth approximation in $W^{2,n}$ via extension and mollification.

\begin{lemma}[Measure estimate for strong supersolutions]
\label{lem-nonhomo-measure-est-strong}
    Let $u\in W^{2,n}(B_1)$ be an $L^n$-strong supersolution of
    \begin{equation*}
        \mathcal{M}_{\lambda,1}^-(D^2u)\leq f
        \quad \text{in $B_1$},
    \end{equation*}
    where $f$ and $\lambda$ are measurable functions in $B_1$
    and $0<\lambda\leq 1$ a.e.
    For $a>0$ and a measurable set $F\subset B_1 \setminus A_a(u)$,
    take the tangent paraboloids of opening $-2a$ to
    $\Gamma_u^a$ on $F$, and slide them up until they touch
    $u$ on a set $E$.
    If $E\Subset B_1$, then there exists a constant $C>0$,
    depending only on $n$, such that
    \begin{equation*}
        |F|\leq C\int_{A_{2a}(u)\setminus A_a(u)}
        \left[
            \lambda^{-(n-1)}
            +\lambda^{-(n-1)}\frac{f^+}{a\lambda}
            +\left(\frac{f^+}{a\lambda}\right)^n
        \right]\,\mathrm{d}x.
    \end{equation*}
\end{lemma}

\begin{proof}
By inner approximation, we may assume that $F$ is compact.
Set $v=u/a+|x|^2/2$, so that $A_a(u)=A_0(v)$ and
$A_{2a}(u)=A_1(v)$.
Let $V$ be the set of vertices of the paraboloids of opening $-1$
touching $\Gamma_v^0$ from below on $F$.
As in the proof of Lemma~\ref{lem-nonhomo-measure-est},
both $V$ and the corresponding contact set $E$ are compact,
and Lemma~\ref{lem-sliding-strong} gives $|E\cap A_0(v)|=0$.
Since $E\subset A_1(v)$, we obtain
\begin{equation}\label{eq-strong-contact-difference}
    E\subset A_{2a}(u)\setminus A_a(u)
    \quad \text{up to a set of measure zero}.
\end{equation}

By the Sobolev extension theorem and mollification, together
with the Sobolev embedding, we choose a sequence of functions
$\{u_j\}$, smooth in a neighborhood of $\overline B_1$, such that
\begin{equation*}
    u_j\to u
    \quad \text{in $W^{2,n}(B_1)$ and uniformly on $\overline B_1$}.
\end{equation*}
Set $v_j=u_j/a+|x|^2/2$, and let $E_j$ be the contact set
obtained by sliding the paraboloids of opening $-1$ with
vertices in $V$ from below $v_j$.
Uniform convergence and compactness imply that, for every
sufficiently small $\delta>0$,
\begin{equation}\label{eq-strong-contact-stability}
    E_j\subset E^\delta\coloneqq
    \{x\in B_1:\dist(x,E)<\delta\}\Subset B_1
    \quad \text{for all sufficiently large $j$}.
\end{equation}
For such $j$ and each $x\in E_j$, there exists $y\in V$
such that $z\mapsto v_j(z)+|z-y|^2/2$ attains its minimum
on $\overline B_1$ at $x$.
Since $x\in B_1$ and $v_j$ is smooth, we have
\begin{equation*}
    x+\nabla v_j(x)=y
    \quad \text{and} \quad
    I+D^2v_j(x)\geq 0.
\end{equation*}
Thus, the smooth map $m_j(x)=x+\nabla v_j(x)$ maps $E_j$
onto $V$, and the area formula gives
\begin{equation*}
|V|\leq\int_{E_j}\det(I+D^2v_j)\,\mathrm{d}x.
\end{equation*}

By recalling the algebraic inequality
\begin{equation*}
    |\det A-\det B| \leq C(n) (\|A\|+\|B\|)^{n-1}\|A-B\| \quad \text{for any $A, B \in \mathbb{R}^{n \times n}$},
\end{equation*}
we use H\"older's inequality and the convergence $v_j \to v$ in $W^{2, n}(B_1)$ to find that
\begin{equation}\label{eq-strong-determinant-convergence}
\begin{aligned}
    &\|\det(I+D^2v_j)-\det(I+D^2v)\|_{L^1(B_1)}\\
    &\qquad\leq C(n)
    \left(
        \|I+D^2v_j\|_{L^n(B_1)}
        +\|I+D^2v\|_{L^n(B_1)}
    \right)^{n-1}
    \|D^2v_j-D^2v\|_{L^n(B_1)}
    \to 0.
\end{aligned}
\end{equation}
For fixed $\delta>0$, combining
\eqref{eq-strong-contact-stability}--\eqref{eq-strong-determinant-convergence}
and letting $j\to\infty$, we obtain
\begin{equation*}
|V|\leq\lim_{j\to\infty}\int_{E^\delta}|\det(I+D^2v_j)|\,\mathrm{d}x=\int_{E^\delta}|\det(I+D^2v)|\,\mathrm{d}x.
\end{equation*}
Since $E$ is compact and $\det(I+D^2v)\in L^1(B_1)$,
letting $\delta\to0^+$ by the dominated convergence theorem yields
\begin{equation*}
    |V|\leq\int_E|\det(I+D^2v)|\,\mathrm{d}x.
\end{equation*}

Since $E\subset A_1(v)$ and $u$ is a strong supersolution, the eigenvalue estimates in the proof of Lemma~\ref{lem-nonhomo-measure-est} now yield that
\begin{equation*}
    0\leq\det(I+D^2v)\leq C(n)
    \left[
        \lambda^{-(n-1)}
        +\lambda^{-(n-1)}\frac{f^+}{a\lambda}
        +\left(\frac{f^+}{a\lambda}\right)^n
    \right]
    \quad \text{a.e. on $E$},
\end{equation*}
and hence
\begin{equation*}
    |V|\leq C(n)\int_E
    \left[
        \lambda^{-(n-1)}
        +\lambda^{-(n-1)}\frac{f^+}{a\lambda}
        +\left(\frac{f^+}{a\lambda}\right)^n
    \right]\,\mathrm{d}x.
\end{equation*}
Together with \eqref{eq-strong-contact-difference} and
$|F|\leq|V|$ from Lemma~\ref{lem-mooney}, this proves the lemma.
\end{proof}

\begin{proof}[Proof of Theorem~\ref{thm-log-hessian-estimate-strong}]
We first assume that $u\in W^{2,n}(B_1)$, so that $u\in C(\overline{B}_1)$ by the Sobolev embedding. Under the normalization \eqref{eq-global-normalization},
we define the contact sets as in Step 1 of the proof of Theorem~\ref{thm-log-hessian-estimate}. The same argument gives \eqref{eq-global-interior-contact}.
Applying Lemma~\ref{lem-nonhomo-measure-est-strong} instead of Lemma~\ref{lem-nonhomo-measure-est}, we obtain \eqref{eq-global-measure}. The same summation and normalization arguments then prove the estimate.

For a general
$u\in W^{2,n}_{\mathrm{loc}}(B_1)\cap L^\infty(B_1)$ and $r\in(1/2,1)$, let $ u_r(x)=u(rx)$, $\lambda_r(x)=\lambda(rx)$ and $f_r(x)=r^2f(rx)$. Then $u_r\in W^{2,n}(B_1)$ is an $L^n$-strong supersolution with ellipticity $\lambda_r$ and right-hand side $f_r$. Applying the preceding estimate to $u_r$ and letting $r\to1^-$, we conclude by the monotone convergence theorem, as in Step 2 of the proof of Theorem~\ref{thm-log-hessian-estimate}.
\end{proof}

We conclude this section with the $L^n$-strong versions of
\cite[Lemma~4.1 and Proposition~1.4]{Bow26}.
The $L^s$-strong solution version in \cite[Theorem~6.1]{Bow26}
assumes $s>n$ to justify the use of the area formula.
At $s=n$, we justify this step by smooth approximation in $W^{2,n}$, as in the proof of Lemma~\ref{lem-nonhomo-measure-est-strong}.
For the following two results, we use the notation
\begin{equation*}
    \kappa\coloneqq\frac{p}{p-(n-1)}
    \quad\text{and}\quad
    \Gamma_p(B)\coloneqq
    \left(\fint_B\lambda^{-p}\,\mathrm{d}x\right)^{1/p}
    \quad \text{for } p>n-1.
\end{equation*}

We first prove the $L^n$-strong version of \cite[Lemma~4.1]{Bow26}.
The key idea is to apply the area formula for smooth approximations and pass to the limit.

\begin{lemma}\label{lem-log-measure-strong}
Let $\lambda$ be measurable in $B_2$ with
$0<\lambda\leq1$ a.e. and
$\lambda^{-1}\in L^p(B_2)$ for some $p>n-1$.
Suppose that $u\in W^{2,n}_{\mathrm{loc}}(B_2)$ is a
nonnegative $L^n$-strong subsolution of
\begin{equation*}
    \mathcal{M}_{\lambda,1}^+(D^2u)\geq0
    \quad\text{in $B_2$},
\end{equation*}
and that $u(0)=2$.
Then, for every $q>n-1$, there exists a constant $C>0$,
depending only on $n$, $p$ and $q$, such that
\begin{equation*}
    1\leq C\Gamma_p(B_2)^{n-1}
    \left(
        \int_{B_2\cap\{u>1\}}(\log u)^{q\kappa}\,\mathrm{d}x
    \right)^{1/\kappa}.
\end{equation*}
\end{lemma}

\begin{proof}
We use the double exponential blow-up barrier from \cite[Lemma~4.1]{Bow26}:
\begin{equation*}
    U(x)\coloneqq
    c_0\exp\left(\exp\left(\frac{1}{1-|x|^2}\right)\right),
    \quad
    \text{where }c_0=2\exp(-\exp(16/15)),
\end{equation*}
so that $U(0)>1$, $U\leq2$ on $\overline B_{1/4}$,
and $U\to\infty$ as $|x|\to1$.
Moreover, $D^2U\geq c_1I$ for some $c_1>0$, and
$DU:B_1\to\mathbb{R}^n$ is a smooth bijection.
In particular, $(DU)^{-1}$ is smooth on $\mathbb{R}^n$ with bounded derivative.

By mollification, we choose smooth functions $\{u_j\}$ such that $u_j(0)=2$ and
\begin{equation*}
    u_j\to u
    \quad\text{in $W^{2,n}(B_{3/2})$ and uniformly on $\overline B_{3/2}$}.
\end{equation*}
For each $y\in V \coloneqq \overline B_{1/4}$, maximize $u_j(z)-U(z-y)$ over $B_1(y)$,
and let $E_j$ be the set of all corresponding contact points; define $E$ in the same way for $u$. Then we observe that 
\begin{equation*}
    u_j(x)-U(x-y)
    \geq u_j(0)-U(-y)
    =2-U(y)\geq0
\end{equation*}
at every contact point $x\in E_j$ with center $y\in V$; the same inequality holds for $u$ on $E$.

The uniform boundedness of $u_j$ and the blow-up of $U$ on $\partial B_1$ show that $E_j, E \Subset B_{5/4}$ for all sufficiently large $j$. Moreover, uniform convergence
and compactness again imply that, for every sufficiently small $\delta>0$,
\begin{equation}\label{eq-log-contact-stability}
    E_j\subset E^\delta\coloneqq
    \{x\in B_{3/2}:\dist(x,E)<\delta\}\Subset B_{3/2}
    \quad\text{for all sufficiently large $j$}.
\end{equation}

At a contact point $x\in E_j$ with center $y\in V$, the contact condition gives
\begin{equation*}
    Du_j(x)=DU(x-y).
\end{equation*}
Thus, the smooth map
\begin{equation*}
    m_j(x)\coloneqq x-(DU)^{-1}(Du_j(x))
\end{equation*}
maps $E_j$ onto $V$.
We also let
\begin{equation*}
    m(x)\coloneqq x-(DU)^{-1}(Du(x)).
\end{equation*}
By differentiating $Du_j(x)=DU(x-m_j(x))$, we have
\begin{equation*}
    D^2u_j(x)=D^2U(x-m_j(x))(I-Dm_j(x)) \quad \text{and so} \quad Dm_j=I-(D^2U)^{-1}D^2u_j \in L^n(B_{3/2}).
\end{equation*}

We now claim that, after passing to a subsequence,
\begin{equation*}
    m_j\to m \quad \text{in $W^{1,n}(B_{3/2})$}.
\end{equation*}
Indeed, the convergence in $L^n(B_{3/2})$ follows from the
Lipschitz continuity of $(DU)^{-1}$.
Passing to a subsequence, we may assume that $Du_j\to Du$ (and so $m_j \to m$)
a.e. in $B_{3/2}$. Set
\begin{equation*}
    A_j(x)\coloneqq[D^2U(x-m_j(x))]^{-1},
    \quad \text{and} \quad
    A(x)\coloneqq[D^2U(x-m(x))]^{-1}.
\end{equation*}
Then $A_j\to A$ a.e., and $D^2U\geq c_1I$ implies $\|A_j\|,\|A\|\leq \sqrt{n} c_1^{-1}$.
Since 
\begin{equation*}
    Dm_j-Dm=-A_j(D^2u_j-D^2u)-(A_j-A)D^2u,
\end{equation*}
we have
\begin{equation*}
        \|Dm_j-Dm\|_{L^n(B_{3/2})}
        \leq \sqrt{n} c_1^{-1}\|D^2u_j-D^2u\|_{L^n(B_{3/2})}+\|(A_j-A)D^2u\|_{L^n(B_{3/2})}
        \to 0,
\end{equation*}
where the last term tends to zero by the dominated convergence theorem.

Therefore, as in the proof of Lemma~\ref{lem-nonhomo-measure-est-strong}, we obtain
\begin{equation*}
    \|\det Dm_j-\det Dm\|_{L^1(B_{3/2})} \to0 \quad \text{as $j \to \infty$}.
\end{equation*}
Since $m_j(E_j)=V$, the area formula and
\eqref{eq-log-contact-stability} yield
\begin{equation*}
    |V|
    \leq\int_{E_j}|\det Dm_j|\,\mathrm{d}x
    \leq\int_{E^\delta}|\det Dm_j|\,\mathrm{d}x
\end{equation*}
for all sufficiently large $j$.
Letting first $j\to\infty$ and then $\delta\to0^+$,
we arrive at
\begin{equation}\label{eq-local-maximum-contact-measure}
    |V|\leq\int_E|\det Dm|\,\mathrm{d}x.
\end{equation}

Since a strong subsolution $u\in W^{2,n}_{\mathrm{loc}}(B_2)$ admits
a second-order expansion a.e., the eigenvalue estimates in the proof of
\cite[Lemma~4.1]{Bow26} yield
\begin{equation*}
    |\det Dm(x)|
    \leq C\lambda(x)^{-(n-1)}(\log u(x))^q
    \quad\text{a.e. on $E$},
\end{equation*}
where we use $u(x)\geq U(x-m(x))>1$ on $E$. Combining this with \eqref{eq-local-maximum-contact-measure} and H\"older's inequality, we find that
\begin{equation*}
\begin{aligned}
    |V|
    &\leq C\int_E
    \lambda^{-(n-1)}(\log u)^q\,\mathrm{d}x\\
    &\leq C
    \left(\int_{B_2}\lambda^{-p}\,\mathrm{d}x\right)^{(n-1)/p}
    \left(
        \int_{B_2\cap\{u>1\}}
        (\log u)^{q\kappa}\,\mathrm{d}x
    \right)^{1/\kappa}.
\end{aligned}
\end{equation*}
Since $|V|=|B_{1/4}|$, the definition of $\Gamma_p$
gives the conclusion.
\end{proof}

We now obtain the $L^n$-strong solution version of
\cite[Proposition~1.4]{Bow26}.
\begin{proposition}[Logarithmic local maximum principle]
\label{prop-log-local-maximum-strong}
Let $\lambda$ be measurable in $B_3$ with
$0<\lambda\leq1$ a.e. and
$\lambda^{-1}\in L^p(B_3)$ for some $p>n-1$.
Suppose that $u\in W^{2,n}_{\mathrm{loc}}(B_3)$ is a
nonnegative $L^n$-strong subsolution of
\begin{equation*}
    \mathcal{M}_{\lambda,1}^+(D^2u)\geq0
    \quad\text{in $B_3$}.
\end{equation*}
Fix $q>n-1$. If $\sup_{B_1}u>2$ and $\tau>\kappa q$, then
\begin{equation*}
    \sup_{B_1}u\leq2\exp\left[
        C\Gamma_p(B_3)^{\frac{\kappa(n-1)}{\tau-\kappa q}}
        \left(
            \int_{B_3\cap\{u>1\}}
            (\log u)^\tau\,\mathrm{d}x
        \right)^{\frac{1}{\tau-\kappa q}}
    \right],
\end{equation*}
where $C>0$ depends only on $n$, $p$, $q$ and $\tau$.
\end{proposition}

\begin{proof}
The conclusion follows by the same argument as in the proof of \cite[Proposition~1.4]{Bow26}, using Lemma~\ref{lem-log-measure-strong} instead of \cite[Lemma~4.1]{Bow26}.
\end{proof}

\section{Grushin-type operators}\label{sec-application}
In this section, we apply Theorem~\ref{thm-log-hessian-estimate-strong}
and Proposition~\ref{prop-log-local-maximum-strong} to study the regularity
of solutions to Grushin-type equations. For $x=(y,z)\in\mathbb{R}^{n_1}\times\mathbb{R}^{n_2}$ with
$n_1,n_2\geq 1$ and $n=n_1+n_2$, we let $G_{\gamma}$ be the
Grushin-type operator defined by
\begin{equation*}
    G_{\gamma}u\coloneqq |y|^{\gamma}\Delta_yu+\Delta_zu.
\end{equation*}
We are interested in the regularity of $L^n$-strong solutions
$u\in W^{2,n}_{\mathrm{loc}}(B_1)$ of
\begin{equation*}
    G_{\gamma}u=0
    \quad \text{in $B_1$}.
\end{equation*}
We note that when $\gamma<0$, the equation is understood in the
normalized form
\begin{equation*}
    \Delta_yu+|y|^{-\gamma}\Delta_zu=0
    \quad \text{in $B_1$},
\end{equation*}
so that $\Lambda\equiv 1$.
In particular, the normalized ellipticity is given by $\lambda_{\gamma}(y)\coloneqq |y|^{|\gamma|}$, and
\begin{equation*}
    \int_{B_1}\lambda_{\gamma}^{-p}\,\mathrm{d}x
    =n_1|B_1^{n_1}||B_1^{n_2}|
    \int_0^1
    r^{n_1-1-|\gamma|p}(1-r^2)^{n_2/2}\,\mathrm{d}r.
\end{equation*}
Thus, we have $\lambda_{\gamma}^{-1}\in L^p(B_1)$ if and only if
$|\gamma|p<n_1$.

\begin{proposition}\label{prop-Grushin-strong}
Let $n=n_1+n_2$, and suppose that
\begin{equation*}
    |\gamma|<\frac{n_1}{n(n-1)}.
\end{equation*}
Let $u\in W^{2,n}_{\mathrm{loc}}(B_1) \cap L^{\infty}(B_1)$ be an $L^n$-strong solution of
\begin{equation*}
    G_{\gamma}u=0
    \quad \text{in $B_1$}.
\end{equation*}
Then we have $u\in C_{z,\mathrm{loc}}^{\infty}(B_1)$ and
\begin{itemize}
    \item if $\gamma=0$, then $u\in C_{\mathrm{loc}}^{\infty}(B_1)$;
    \item if $\gamma<0$, then $u\in C_{\mathrm{loc}}^{2,-\gamma}(B_1)$;
    \item if $\gamma>0$, then $u\in C_{\mathrm{loc}}^{1,\beta}(B_1)$ for every $\beta \in (0, 1-\gamma)$.
\end{itemize}
\end{proposition}

\begin{proof}
Since $u$ is harmonic when $\gamma=0$, we assume that $\gamma\neq 0$. \medskip

(Step 1. $D_{zz}^2u\in L_{\mathrm{loc}}^{\infty}(B_1)$.)
We fix an exponent $p$ satisfying
\begin{equation*}
    n(n-1)<p<\frac{n_1}{|\gamma|},
\end{equation*}
and set
\begin{equation*}
    \tau\coloneqq\frac{p}{n-1}
    \quad \text{and} \quad
    \kappa\coloneqq\frac{p}{p-(n-1)}.
\end{equation*}
Since $p>n(n-1)$, we also fix an exponent $q$ satisfying
\begin{equation*}
    n-1<q<\frac{p-(n-1)}{n-1}
\end{equation*}
so that $\tau-\kappa q>0$. These choices of $p$, $q$, $\tau$ and $\kappa$ allow us to apply Theorem~\ref{thm-log-hessian-estimate-strong} and Proposition~\ref{prop-log-local-maximum-strong} below.

For $e\in\{0\}\times\mathbb{R}^{n_2}$ with $|e|=1$ and
$0<|h|<1/4$, we define the second difference quotient
\begin{equation*}
    v_{h,e}(x)\coloneqq
    \frac{u(x+he)+u(x-he)-2u(x)}{h^2} \quad \text{for $x\in B_{1/2}$},
\end{equation*}
and set
\begin{equation*}
    w_{h,e}\coloneqq\sqrt{1+v_{h,e}^2} \quad \text{and} \quad M_{h,e}\coloneqq\sup_{B_{1/6}}w_{h,e}.
\end{equation*}
Then we observe the following.
\begin{enumerate}[(i)]
    \item Since $\|D^2w_{h, e}\| \leq \|D^2v_{h, e}\|+|Dv_{h, e}|^2$, we have $w_{h, e} \in W^{2, n}_{\mathrm{loc}}(B_{1/2})$.
    
     \item Since the coefficients of $G_{\gamma}$ are independent of $z$, we have $G_{\gamma}v_{h,e}=0$ in $B_{1/2}$. Hence, it follows from a direct calculation that
     \begin{equation*}
G_{\gamma}w_{h,e}=\frac{|y|^{\gamma}|D_yv_{h,e}|^2+|D_zv_{h,e}|^2}
    {(1+v_{h,e}^2)^{3/2}}
    \geq 0
    \quad \text{a.e. in $B_{1/2}$}.
\end{equation*}
In particular, $w_{h,e}$ is a nonnegative $L^n$-strong subsolution of $\mathcal{M}_{\lambda_{\gamma},1}^+(D^2w_{h,e})\geq0$ in $B_{1/2}$.
    
    \item Since $1\leq w_{h,e}(x)\leq 1+|v_{h,e}(x)|
        \leq 1+\max\{\underline{\Theta}_u(x;B_1),\overline{\Theta}_u(x;B_1)\}$,
   we apply Theorem~\ref{thm-log-hessian-estimate-strong} to $\pm u$ to obtain
\begin{equation*}
    \begin{aligned}
        \int_{B_{1/2}}(\log w_{h,e})^{\tau}\,\mathrm{d}x
       &\leq
        \int_{B_{1}}
        \log^{\tau}(1+\underline{\Theta}_u(x;B_1))\,\mathrm{d}x
        +\int_{B_{1}}
        \log^{\tau}(1+\overline{\Theta}_u(x;B_1))\,\mathrm{d}x\\
        &\leq
        C\left(
        1+\|\lambda_{\gamma}^{-1}\|_{L^p(B_1)}^p
        +\log^{\tau}(1+\|u\|_{L^{\infty}(B_1)})
        \right).
    \end{aligned}
\end{equation*}
\end{enumerate}

We are now ready to prove the desired Hessian estimates with respect to the $z$-variables. If $M_{h,e}\leq 2$, there is nothing to prove. Otherwise, applying Proposition~\ref{prop-log-local-maximum-strong} to $w_{h,e}$ on the rescaled ball gives
\begin{equation*}
    \begin{aligned}
        M_{h,e}
        &\leq
        2\exp\left[
        c\Gamma_p(B_{1/2})^{\frac{\kappa(n-1)}{\tau-\kappa q}}
        \left(
        \int_{B_{1/2}}
        (\log w_{h,e})^{\tau}\,\mathrm{d}x
        \right)^{\frac{1}{\tau-\kappa q}}
        \right]\\
        &\leq K=K(n, p, q, \|\lambda_{\gamma}^{-1}\|_{L^p(B_1)}, \|u\|_{L^{\infty}(B_1)}).
    \end{aligned}
\end{equation*}
In particular, $K$ is independent of $h$, $e$, and the $W^{2,n}$-norm of $u$.
Since $|v_{h,e}|\leq w_{h,e}$, we conclude that
\begin{equation*}
    \sup_{B_{1/6}}|v_{h,e}|\leq K.
\end{equation*}
Letting $h\to 0$ and applying a standard covering argument, we conclude that
\begin{equation*}
    D_{zz}^2u\in L_{\mathrm{loc}}^{\infty}(B_1).
\end{equation*}

(Step 2. $u\in C_{z,\mathrm{loc}}^{\infty}(B_1)$.)
We use the bootstrap argument with the aid of Step 1. For $e\in\{0\}\times\mathbb{R}^{n_2}$ with $|e|=1$, let
\begin{equation*}
    \delta_{h,e}u(x)\coloneqq
    \frac{u(x+he)-u(x)}{h}.
\end{equation*}
Again, $\delta_{h,e}u$ is an $L^n$-strong solution of $G_{\gamma}(\delta_{h,e}u)=0$ in $U \Subset B_1$, provided that $|h|$ is sufficiently small.
By Step 1 and interpolation in the $z$-variables, we have
\begin{equation*}
    \|\delta_{h,e}u\|_{L^{\infty}(U)} \leq K_U,
\end{equation*}
where $K_U$ is independent of $h$ and $e$. Therefore, applying Step 1 to $\delta_{h,e}u$ on smaller domains, we obtain
\begin{equation*}
    \sup_{e\in\{0\}\times S^{n_2-1}}
    \sup_{0<|h|<h_0}
    \|D_{zz}^2(\delta_{h,e}u)\|_{L^{\infty}(U')}<\infty
    \quad \text{whenever $U'\Subset U\Subset B_1$},
\end{equation*}
where $h_0>0$ is sufficiently small.
Letting $h\to 0$, we find that
\begin{equation*}
    D_{zzz}^3u\in L_{\mathrm{loc}}^{\infty}(B_1).
\end{equation*}
Repeating this argument with iterated first difference quotients on smaller domains gives
\begin{equation*}
    D_z^{\mu}u\in L_{\mathrm{loc}}^{\infty}(B_1)
    \quad \text{for every multi-index $\mu$ involving only the $z$-variables}.
\end{equation*}

(Step 3. Regularity in the remaining directions.)
We may write the equation as
\begin{equation}\label{eq-Grushin-Poisson}
    \Delta u=(1-|y|^{-\gamma})\Delta_zu,
\end{equation}
and distinguish the following two cases.

\noindent\emph{Case 1: $\gamma<0$.}
Let $\mu$ be a multi-index involving only the $z$-variables, and
set $v=D_z^{\mu}u$.
Differentiating \eqref{eq-Grushin-Poisson} in the $z$-variables and using Step 2,
we obtain
\begin{equation*}
    \Delta v=(1-|y|^{-\gamma})\Delta_zv
    \in L_{\mathrm{loc}}^{\infty}(B_1).
\end{equation*}
Thus, the interior Calder\'on--Zygmund estimates and the Morrey embedding imply that $ v=D_z^{\mu}u\in C_{\mathrm{loc}}^{1,\alpha}(B_1)$ for every $\alpha \in (0,1)$. In particular, since $-\gamma \in (0,1)$, we have
\begin{equation*}
    (1-|y|^{-\gamma})\Delta_zu \in C_{\mathrm{loc}}^{0,-\gamma}(B_1).
\end{equation*}
The interior Schauder estimate applied to
\eqref{eq-Grushin-Poisson} therefore gives
$u\in C_{\mathrm{loc}}^{2,-\gamma}(B_1)$.

\noindent\emph{Case 2: $\gamma>0$.}
Let $\mu$ be a multi-index involving only the $z$-variables, and set $v=D_z^{\mu}u$. Differentiating \eqref{eq-Grushin-Poisson} in the $z$-variables and using Step 2, we obtain
\begin{equation*}
    \Delta v=(1-|y|^{-\gamma})\Delta_zv \in L^s_{\mathrm{loc}}(B_1) \quad \text{for every $s \in (1, n_1/\gamma)$}.
\end{equation*}
Thus, the interior Calder\'on--Zygmund estimates imply that $v=D_z^{\mu}u\in W_{\mathrm{loc}}^{2,s}(B_1)$ for every $s \in (1, n_1/\gamma)$. If we  choose $|\mu|=1$ and $s>n$, then the Morrey embedding shows that $D_{z_j}u \in C^{1,\alpha}_{\mathrm{loc}}(B_1)$ for some $\alpha \in (0,1)$; in particular, we have
\begin{equation*}
    D_yD_zu\in L_{\mathrm{loc}}^{\infty}(B_1).
\end{equation*}
Together with Step 1, this implies that $D_zu$ is locally
Lipschitz in all variables, and $D_yu$ is locally Lipschitz in
the $z$-variables.

To obtain further regularity of $D_yu$ in the $y$-variables, we observe that
\begin{equation*}
    \Delta_yu(\cdot,z)
    =-|y|^{-\gamma}\Delta_zu(\cdot,z)
    \quad \text{in $B_{2r}^{n_1}$} \quad \text{for each fixed $z \in B_{2r}^{n_2}$.}
\end{equation*}
The interior Calder\'on--Zygmund estimates in $\mathbb{R}^{n_1}$
and the corresponding Morrey embedding yield that $u \in W^{2, s}_{y,\mathrm{loc}}(B_{2r}^{n_1})\subset C^{1, 1-n_1/s}_{y,\mathrm{loc}}(B_{2r}^{n_1})$ for $s \in (n_1, n_1/\gamma)$ with the estimate
\begin{equation*}
    \sup_{z\in B_r^{n_2}}
    \|u(\cdot,z)\|_{C_y^{1,\beta}(B_r^{n_1})}<\infty
    \quad \text{for every $\beta \in (0, 1-\gamma)$}.
\end{equation*}
Combining these estimates with the local Lipschitz bounds above, we conclude that
\begin{equation*}
    u\in C_{\mathrm{loc}}^{1,\beta}(B_1)
    \quad \text{for every $\beta \in (0, 1-\gamma)$}. \qedhere
\end{equation*}
\end{proof}

\begin{remark}[Viscosity solutions of Grushin equations]\label{rmk-Grushin-viscosity}
Let $u \in C(B_1)$ be a $C$-viscosity solution of $G_{\gamma}u=0$ in $B_1$ with $|\gamma|<n_1/(n(n-1))$. Suppose that there exists a family of functions $\{u_{\varepsilon}\}_{\varepsilon \in (0,1)} \subset C(B_1)$ such that 
\begin{itemize}
    \item $u_{\varepsilon}\to u$ locally uniformly as $\varepsilon\to0$;

    \item $u_{\varepsilon}$ is a $C$-viscosity solution of the regularized Grushin equation
\begin{equation*}
    G_{\gamma,\varepsilon}u_{\varepsilon}\coloneqq
    \left(\frac{|y|^2+\varepsilon^2}{1+\varepsilon^2}\right)^{\gamma/2}\Delta_yu_{\varepsilon}+\Delta_zu_{\varepsilon}=0.
\end{equation*}
\end{itemize}

Then $u$ satisfies the same regularity results obtained in Proposition~\ref{prop-Grushin-strong}. Indeed, since the normalized ellipticity satisfies
\begin{equation*}
    \lambda_{\gamma,\varepsilon}(y)
    =\left(\frac{|y|^2+\varepsilon^2}{1+\varepsilon^2}\right)^{|\gamma|/2}
    \geq |y|^{|\gamma|},
\end{equation*}
$\|\lambda_{\gamma,\varepsilon}^{-1}\|_{L^p(B_1)}$ is bounded independently of $\varepsilon$ whenever $n(n-1)<p<n_1/|\gamma|$. Moreover, local uniform convergence implies that $\{u_\varepsilon\}$ is locally uniformly bounded as $\varepsilon\to0$. We repeat the proof of Proposition~\ref{prop-Grushin-strong} by using Theorem~\ref{thm-log-hessian-estimate} and \cite[Proposition~1.4]{Bow26} instead of their strong-solution counterparts, Theorem~\ref{thm-log-hessian-estimate-strong} and Proposition~\ref{prop-log-local-maximum-strong}. Thus, the same local regularity estimates hold uniformly as $\varepsilon\to0$ and the conclusion follows by passing to the limit.

On the other hand, the approximation assumption cannot be omitted, since these
regularity results need not hold for arbitrary $C$-viscosity solutions of the Grushin equation. Indeed, let $n_1=1$ and $0<\gamma<1/(n(n-1))$.
The function $u(y,z)=|y|$ is a $C$-viscosity solution of
\begin{equation*}
    |y|^\gamma u_{yy}+\Delta_z u=0;
\end{equation*}
see \cite{KL25, Bow26} for details. However, it is not $C^1$ across $\{y=0\}$ and therefore does not satisfy the conclusion of Proposition~\ref{prop-Grushin-strong}.
In particular, we conclude that this solution cannot be obtained as a locally
uniform limit of solutions to the above regularized equations.
\end{remark}

\begin{remark}[Sharpness of Proposition~\ref{prop-Grushin-strong}]
\label{rmk-Grushin-exponent}
The following example illustrates the essential dependence
of regularity on the sign of $\gamma$, and shows that the
H\"older exponents in Proposition~\ref{prop-Grushin-strong}
are sharp for $\gamma<0$ and nearly sharp for $\gamma>0$.

For $0<|\gamma|<n_1/(n(n-1))$, consider
\begin{equation*}
    w_{\gamma}(y,z)
    \coloneqq
    \frac{z_1^2}{2}
    -\frac{|y|^{2-\gamma}}{(2-\gamma)(n_1-\gamma)}.
\end{equation*}
Then we have $\Delta_yw_{\gamma}=-|y|^{-\gamma}$ and $\Delta_zw_{\gamma}=1$ for $y \neq 0$, and hence $G_{\gamma}w_{\gamma}=0$ a.e. Since $w_{\gamma}\in W_{\mathrm{loc}}^{2,n}(B_1)$, it follows that $w_{\gamma}$ is an $L^n$-strong solution of the Grushin equation covered by Proposition~\ref{prop-Grushin-strong}.

The function $w_{\gamma}$ is smooth in $z$.
When $\gamma<0$, it belongs to $C_{\mathrm{loc}}^{2,-\gamma}$
and has no higher H\"older exponent for its second derivatives
across $\{y=0\}$.
When $\gamma>0$, it belongs to $C_{\mathrm{loc}}^{1,1-\gamma}$
but not to $C_{\mathrm{loc}}^{1,\beta}$ for any
$\beta \in (1-\gamma, 1)$.
\end{remark}

\begin{remark}[Regularity in invariant directions]
\label{rmk-invariant-directions}
The argument in Steps 1 and 2 of
Proposition~\ref{prop-Grushin-strong} also applies to linear
equations whose coefficients are independent of some variables.
More precisely, let
$x=(y,z)\in\mathbb{R}^{n_1}\times\mathbb{R}^{n_2}$ and consider
\begin{equation*}
    Lu\coloneqq\tr(A(y)D^2u)=0 \quad \text{in $B_1$},
\end{equation*}
where $A(y)$ is a measurable symmetric matrix satisfying
\begin{equation*}
    0<\lambda(y)I\leq A(y)\leq I,
    \quad
    \lambda^{-1}\in L^p(B_1),
    \quad p>n(n-1).
\end{equation*}
If $u\in W_{\mathrm{loc}}^{2,n}(B_1)$ is an
$L^n$-strong solution of $Lu=0$, then the same argument as in the proof of Proposition~\ref{prop-Grushin-strong} shows that
\begin{equation*}
    D_z^{\mu}u\in L_{\mathrm{loc}}^{\infty}(B_1)
    \quad \text{for every multi-index $\mu$ involving only the $z$-variables}.
\end{equation*}
Further regularity in the $y$-variables requires additional structure on $A$.
\end{remark}


\begin{thebibliography}{10}

\bibitem{Ale66}
A.~D. Aleksandrov.
\newblock Majorants of solutions of linear equations of order two.
\newblock {\em Vestnik Leningrad. Univ.}, 21(1):5--25, 1966.

\bibitem{ASS12}
S.~N. Armstrong, L.~E. Silvestre, and C.~K. Smart.
\newblock Partial regularity of solutions of fully nonlinear, uniformly
  elliptic equations.
\newblock {\em Comm. Pure Appl. Math.}, 65(8):1169--1184, 2012.

\bibitem{AS14}
S.~N. Armstrong and C.~K. Smart.
\newblock Regularity and stochastic homogenization of fully nonlinear equations
  without uniform ellipticity.
\newblock {\em Ann. Probab.}, 42(6):2558--2594, 2014.

\bibitem{Bak61}
I.~J. Bakel\cprime~man.
\newblock On the theory of quasilinear elliptic equations.
\newblock {\em Sibirsk. Mat. \v{Z}.}, 2:179--186, 1961.

\bibitem{BS21}
P.~Bella and M.~Sch\"{a}ffner.
\newblock Local boundedness and {H}arnack inequality for solutions of linear
  nonuniformly elliptic equations.
\newblock {\em Comm. Pure Appl. Math.}, 74(3):453--477, 2021.

\bibitem{Bow26}
D.~Bowman.
\newblock Harnack inequality for non-uniformly elliptic equations in
  non-divergence form.
\newblock {\em arXiv preprint arXiv:2604.13303}, 2026.

\bibitem{CCKS96}
L.~Caffarelli, M.~G. Crandall, M.~Kocan, and A.~Swi\polhk{e}ch.
\newblock On viscosity solutions of fully nonlinear equations with measurable
  ingredients.
\newblock {\em Comm. Pure Appl. Math.}, 49(4):365--397, 1996.

\bibitem{Caf89}
L.~A. Caffarelli.
\newblock Interior a priori estimates for solutions of fully nonlinear
  equations.
\newblock {\em Ann. of Math. (2)}, 130(1):189--213, 1989.

\bibitem{CC95}
L.~A. Caffarelli and X.~Cabr\'{e}.
\newblock {\em Fully nonlinear elliptic equations}, volume~43 of {\em American
  Mathematical Society Colloquium Publications}.
\newblock American Mathematical Society, Providence, RI, 1995.

\bibitem{CZ61}
A.-P. Calder\'on and A.~Zygmund.
\newblock Local properties of solutions of elliptic partial differential
  equations.
\newblock {\em Studia Math.}, 20:171--225, 1961.

\bibitem{CIL92}
M.~G. Crandall, H.~Ishii, and P.-L. Lions.
\newblock User's guide to viscosity solutions of second order partial
  differential equations.
\newblock {\em Bull. Amer. Math. Soc. (N.S.)}, 27(1):1--67, 1992.

\bibitem{DGio57}
E.~De~Giorgi.
\newblock Sulla differenziabilit\`a{} e l'analiticit\`a{} delle estremali degli
  integrali multipli regolari.
\newblock {\em Mem. Accad. Sci. Torino. Cl. Sci. Fis. Mat. Nat. (3)}, 3:25--43,
  1957.

\bibitem{DK11}
H.~Dong and S.~Kim.
\newblock Partial {S}chauder estimates for second-order elliptic and parabolic
  equations.
\newblock {\em Calc. Var. Partial Differential Equations}, 40(3-4):481--500,
  2011.

\bibitem{DK19}
H.~Dong and S.~Kim.
\newblock Partial {S}chauder estimates for second-order elliptic and parabolic
  equations: a revisit.
\newblock {\em Int. Math. Res. Not. IMRN}, (7):2085--2136, 2019.

\bibitem{FKS82}
E.~B. Fabes, C.~E. Kenig, and R.~P. Serapioni.
\newblock The local regularity of solutions of degenerate elliptic equations.
\newblock {\em Comm. Partial Differential Equations}, 7(1):77--116, 1982.

\bibitem{GT01}
D.~Gilbarg and N.~S. Trudinger.
\newblock {\em Elliptic partial differential equations of second order}.
\newblock Classics in Mathematics. Springer-Verlag, Berlin, 2001.
\newblock Reprint of the 1998 edition.

\bibitem{Gru70}
V.~V. Gru{\v s}in.
\newblock A certain class of hypoelliptic operators.
\newblock {\em Mat. Sb. (N.S.)}, 83(125):456--473, 1970.

\bibitem{KL25}
J.~Kim and S.-C. Lee.
\newblock Regularity for solutions of non-uniformly elliptic equations in
  non-divergence form.
\newblock {\em Rev. Mat. Iberoam.}, 41(6):2335--2356, 2025.

\bibitem{KLY24}
T.~Kim, K.-A. Lee, and H.~Yun.
\newblock Generalized {S}chauder theory and its application to
  degenerate/singular parabolic equations.
\newblock {\em Math. Ann.}, 390(4):6049--6109, 2024.

\bibitem{KS79}
N.~V. Krylov and M.~V. Safonov.
\newblock An estimate for the probability of a diffusion process hitting a set
  of positive measure.
\newblock {\em Dokl. Akad. Nauk SSSR}, 245(1):18--20, 1979.

\bibitem{KS80}
N.~V. Krylov and M.~V. Safonov.
\newblock A property of the solutions of parabolic equations with measurable
  coefficients.
\newblock {\em Izv. Akad. Nauk SSSR Ser. Mat.}, 44(1):161--175, 239, 1980.

\bibitem{Lin86}
F.-H. Lin.
\newblock Second derivative {$L^p$}-estimates for elliptic equations of
  nondivergent type.
\newblock {\em Proc. Amer. Math. Soc.}, 96(3):447--451, 1986.

\bibitem{Mal18}
D.~Maldonado.
\newblock On certain degenerate and singular elliptic {PDE}s {I}:
  {N}ondivergence form operators with unbounded drifts and applications to
  subelliptic equations.
\newblock {\em J. Differential Equations}, 264(2):624--678, 2018.

\bibitem{Mon14}
A.~Montanari.
\newblock Harnack inequality for a subelliptic {PDE} in nondivergence form.
\newblock {\em Nonlinear Anal.}, 109:285--300, 2014.

\bibitem{Moo19}
C.~Mooney.
\newblock A proof of the {K}rylov-{S}afonov theorem without localization.
\newblock {\em Comm. Partial Differential Equations}, 44(8):681--690, 2019.

\bibitem{Mos60}
J.~Moser.
\newblock A new proof of {D}e {G}iorgi's theorem concerning the regularity
  problem for elliptic differential equations.
\newblock {\em Comm. Pure Appl. Math.}, 13:457--468, 1960.

\bibitem{Mos61}
J.~Moser.
\newblock On {H}arnack's theorem for elliptic differential equations.
\newblock {\em Comm. Pure Appl. Math.}, 14:577--591, 1961.

\bibitem{NT23}
T.~M. Nascimento and E.~V. Teixeira.
\newblock New regularity estimates for fully nonlinear elliptic equations.
\newblock {\em J. Math. Pures Appl. (9)}, 171:1--25, 2023.

\bibitem{Nas58}
J.~Nash.
\newblock Continuity of solutions of parabolic and elliptic equations.
\newblock {\em Amer. J. Math.}, 80:931--954, 1958.

\bibitem{Puc66}
C.~Pucci.
\newblock Limitazioni per soluzioni di equazioni ellittiche.
\newblock {\em Ann. Mat. Pura Appl. (4)}, 74:15--30, 1966.

\bibitem{Sav07}
O.~Savin.
\newblock Small perturbation solutions for elliptic equations.
\newblock {\em Comm. Partial Differential Equations}, 32(4-6):557--578, 2007.

\bibitem{STV21a}
Y.~Sire, S.~Terracini, and S.~Vita.
\newblock Liouville type theorems and regularity of solutions to degenerate or
  singular problems part {I}: even solutions.
\newblock {\em Comm. Partial Differential Equations}, 46(2):310--361, 2021.

\bibitem{STV21b}
Y.~Sire, S.~Terracini, and S.~Vita.
\newblock Liouville type theorems and regularity of solutions to degenerate or
  singular problems part {II}: odd solutions.
\newblock {\em Math. Eng.}, 3(1):Paper No. 5, 50, 2021.

\bibitem{Tru71}
N.~S. Trudinger.
\newblock On the regularity of generalized solutions of linear, non-uniformly
  elliptic equations.
\newblock {\em Arch. Rational Mech. Anal.}, 42:50--62, 1971.

\bibitem{Tru73}
N.~S. Trudinger.
\newblock Linear elliptic operators with measurable coefficients.
\newblock {\em Ann. Scuola Norm. Sup. Pisa Cl. Sci. (3)}, 27:265--308, 1973.

\bibitem{Wan92}
L.~Wang.
\newblock On the regularity theory of fully nonlinear parabolic equations. {I}.
\newblock {\em Comm. Pure Appl. Math.}, 45(1):27--76, 1992.

\bibitem{Wan92a}
L.~Wang.
\newblock On the regularity theory of fully nonlinear parabolic equations.
  {II}.
\newblock {\em Comm. Pure Appl. Math.}, 45(2):141--178, 1992.

\end{thebibliography}

\end{document}